\documentclass[a4paper,11pt,final]{article}
\usepackage{amsmath,amsfonts,amssymb,amsthm}
\usepackage[hmargin={28mm,28mm},vmargin={30mm,35mm}]{geometry}
\usepackage[ocgcolorlinks,allcolors={blue},breaklinks]{hyperref}
\usepackage{paralist}
\usepackage{xcolor,graphicx}

\usepackage[color,notref,notcite]{showkeys}
\definecolor{labelkey}{rgb}{0.6,0,1}

\usepackage{tagging}
 \usetag{dg}

\usepackage{newtxtext,newtxmath}

\usepackage[bold]{hhtensor}

\newcommand{\email}[1]{\href{mailto:#1}{\tt #1}}

\theoremstyle{theorem}
\newtheorem{theorem}{Theorem}
\newtheorem{lemma}[theorem]{Lemma}
\newtheorem{proposition}[theorem]{Proposition}

\theoremstyle{remark}
\newtheorem{remark}[theorem]{Remark}
\theoremstyle{definition}

\newtheorem{definition}[theorem]{Definition}
\newtheorem{example}[theorem]{Example}

\newcommand{\Real}{\mathbb{R}}
\newcommand{\Natural}{\mathbb{N}}

\newcommand{\Id}[1][d]{\matr{I}_{#1}}

\newcommand{\Poly}[1]{\mathbb{P}^{#1}}

\newcommand{\SCAL}{{\cdot}}
\newcommand{\GRAD}{\vec{\nabla}}
\newcommand{\GRADh}{\vec{\nabla}_h}
\newcommand{\DIV}{\vec{\nabla}{\cdot}}

\newcommand{\norm}[2][]{\|#2\|_{#1}}
\newcommand{\seminorm}[2][]{|#2|_{#1}}

\newcommand{\term}{\mathfrak{T}}

\newcommand{\st}{\,:\,}

\newcommand{\Mh}[1][h]{\mathcal{M}_{#1}}
\newcommand{\Th}[1][h]{\mathcal{T}_{#1}}
\newcommand{\Fh}[1][h]{\mathcal{F}_{#1}}
\newcommand{\Fhi}[1][h]{\mathcal{F}_{#1}^{{\rm i}}}
\newcommand{\Fhb}[1][h]{\mathcal{F}_{#1}^{{\rm b}}}

\newcommand{\normal}{\vec{n}}

\newcommand{\aH}{H}

\newcommand{\aL}{L}

\newcommand{\aXh}{X_h}
\newcommand{\aYh}{Y_h}

\newcommand{\aIh}{I_h}
\newcommand{\aJh}{J_h}
\newcommand{\abil}{a}

\newcommand{\abilh}{\abil_h}
\newcommand{\coer}{\gamma}

\newcommand{\alin}{\ell}

\newcommand{\alinh}{\alin_h}

\newcommand{\Cerr}[2]{\mathcal E_{h}(#1;#2)}
\newcommand{\Cerrdual}[2]{\mathcal E^{\rm d}_{h}(#1;#2)}

\newcommand{\diff}[1][]{\matr{K}_{#1}}
\newcommand{\sdiff}[1][]{K_{#1}}
\newcommand{\vel}{\vec{\beta}}
\newcommand{\reac}{\mu}

\newcommand{\jump}[2][F]{[#2]_{#1}}
\newcommand{\wavg}[2][F]{\{#2\}_{\vec{\omega},#1}}

\newcommand{\polydeg}{\mathsf{l}}

\newcommand{\lproj}[2][h]{\pi_{#1}^{0,#2}}

\newcommand{\dproj}[2][T]{\pi_{\diff,#1}^{1,#2}}

\newcommand{\Fl}[1][T,F]{\mathfrak F_{#1}}
\newcommand{\flgen}{\mathfrak{f}}

\newcommand{\disc}{{\mathcal D}}

\begin{document}

\title{A third Strang lemma and an Aubin-Nitsche trick for schemes in fully discrete formulation}

\author{Daniele A. Di Pietro\thanks{Institut Montpelli\'erain Alexander Grothendieck, Univ. Montpellier, CNRS (France), \email{daniele.di-pietro@umontpellier.fr}}
  \and
  J\'er\^ome Droniou\thanks{School of Mathematical Sciences, Monash University, Melbourne (Australia), \email{jerome.droniou@monash.edu}}
}

\maketitle

\begin{abstract}
  In this work, we present an abstract error analysis framework for the approximation of linear partial differential equation (PDE) problems in weak formulation.
  We consider approximation methods in fully discrete formulation, where the discrete and continuous spaces are possibly not embedded in a common space.
  A proper notion of consistency is designed, and, under a classical inf--sup condition, it is
  shown to bound the approximation error. This error estimate result is in the spirit
  of Strang's first and second lemmas, but applicable in situations not covered by these lemmas (because of a fully discrete approximation space). An improved estimate is also 
  established in a weaker norm, using the Aubin--Nitsche trick.

  We then apply these abstract estimates to an anisotropic heterogeneous diffusion model
  and two classical families of schemes for this model: Virtual Element and Finite Volume methods.
  For each of these methods, we show that the abstract results yield new error estimates with a precise and mild dependency on the local anisotropy ratio.
  A key intermediate step to derive such estimates for Virtual Element Methods is proving optimal approximation properties of the oblique elliptic projector in weighted Sobolev seminorms.
  This is a result whose interest goes beyond the specific model and methods considered here.
  We also obtain, to our knowledge, the first clear notion of consistency for Finite Volume methods, which leads to a generic error estimate involving the fluxes and valid for a wide range of Finite Volume schemes. An important application is the first error estimate for Multi-Point Flux Approximation L and G methods.
  
  \begin{taggedblock}{dg}%
    In the appendix, not included in the published version of this work, we show that classical estimates for discontinuous Galerkin methods can be obtained with simplified arguments using the abstract framework.
  \end{taggedblock}
  \medskip\\
  
  \textbf{Key words.} Strang lemma,
  Consistency,
  Error estimate,
  Aubin-Nitsche trick,
  Virtual Element Methods,
  \tagged{dg}{Discontinuous Galerkin,}
  Finite Volume methods,
  oblique elliptic projector.
  \medskip\\
  \textbf{AMS subject classification.} 65N08, 65N12, 65N15, 65N30.

\end{abstract}



\section{Introduction}

The second Strang lemma \cite{Strang:72,Strang.Fix:08} is probably the most famous error estimate result for Finite Element Methods, and is used as a starting point for the analysis of non-conforming methods in many reference textbooks; see, e.g., \cite{Ern.Guermond:04,Ciarlet:02}.
In recent years, it has been generalised to novel technologies including, e.g., Discontinuous Galerkin (DG) \cite[Section 1.3]{Di-Pietro.Ern:12} and Virtual Element methods (VEM) \cite[Theorem 2]{Cangiani.Manzini.ea:17}.
In a nutshell, given Hilbert spaces $V$ and $V_h$, a bilinear form $a(\cdot,\cdot)$ (resp. $\abilh(\cdot,\cdot)$) and a linear form $\ell(\cdot)$ (resp. $\alinh(\cdot)$) defined on $V$ (resp. $V_h$), and considering the continuous and discrete problems
\[
\mbox{Find $u\in V$ such that}\quad a(u,v)=\ell(v)\quad \forall v\in V
\]
and
\begin{equation}\label{strang.disc}
\mbox{Find $u_h\in V_h$ such that}\quad \abilh(u_h,v_h)=\alinh(v_h)\quad \forall v_h\in V_h,
\end{equation}
the second Strang lemma provides, under boundedness and inf--sup conditions on $\abilh$, a bound on a proper norm of $u-u_h$ in terms of quantities measuring the approximation properties of $V$ by $V_h$ and the consistency of the discrete problem.

This result has two major constraints: 
\begin{enumerate}[(i)]
\item $V$ and $V_h$ must be subspaces of a common space of functions, to ensure that the sum $V+V_h$ is well defined,
\item $a_h(\cdot,\cdot)$ must be extended to $V+V_h$, such that its restriction to $V$ is consistent with $a(\cdot,\cdot)$ in an appropriate way.
\end{enumerate}
The first constraint is not an issue for Finite Element, DG methods or VEM, whose natural unknowns are functions, but it is not satisfied by a number of other methods such as, e.g., Hybrid High-Order \cite{Di-Pietro.Ern.ea:14}, Mimetic Finite Differences \cite{Beirao-da-Veiga.Lipnikov.ea:14},  cell- and face-centred Finite Volume methods (such as Hybrid Mimetic Mixed methods \cite{Droniou.Eymard:06,Eymard.Gallouet.ea:10,Droniou.Eymard.ea:10}). Indeed, in some of these methods, even though certain components of vectors in $V_h$ represent functions on the mesh cells, other components can represent unknowns/functions on the mesh faces, at the mesh vertices, etc.

Even for methods that clearly satisfy (i), the second constraint can raise some challenges. For example, in DG methods, the extension of $a_h$ can often be made only in $V_*+V_h$, where $V_*$ is a strict subspace of $V$; see, e.g., \cite[Section 1.3.3]{Di-Pietro.Ern:12}.
Possible ways of circumventing the difficulties linked to the insufficient regularity of the exact solution have been proposed, e.g., in \cite{Gudi:10} (trimmed error estimates) and, more recently, in \cite{Ern.Guermond:18} (mollified error estimates).
Other difficulties may be inherent to the approach used to construct the discretisation.
In VEM, the discrete bilinear form $\abilh(\cdot,\cdot)$ often contains contributions defined in an algebraic way that make its extension to $V$ not obvious; the Strang-like estimate of \cite{Cangiani.Manzini.ea:17} circumvents this question of extension of $\abilh(\cdot,\cdot)$, at the expense of additional terms, by extending instead the continuous form $a(\cdot,\cdot)$ to $V+V_h$.
Another example can be found in the family of cell-centered Finite Volume methods \cite{Edwards.Rogers:94,Eymard.Gallouet.ea:00,Aavatsmark.Eigestad.ea:08,Droniou:14}: even if the unknowns can be considered, in these methods, as piecewise constant functions on the mesh, and their formulation can be written as \eqref{strang.disc}, the resulting bilinear form $\abilh(\cdot,\cdot)$ is written in a fully discrete form that makes its extension to a space of functions, and the subsequent analysis, more involved.

\smallskip

In this work, we propose a new error analysis framework, for problems written in weak (Petrov--Galerkin) form, that is free from the two constraints mentioned above. The main idea is to estimate a \emph{discrete} approximation error $\aIh u-u_h$, where $\aIh:V\to V_h$ is a well chosen interpolant of functions onto the discrete space.
With this definition of the approximation error, the exact solution $u$ \emph{need not} be plugged into the discrete bilinear form to write the error equation. Instead, its role is played by $\aIh u$.
The discrete approximation error can then be estimated solely in terms of the (discrete) norm in $V_h$ under a \emph{stability} assumption on $\abilh$ (an inf--sup condition) and in terms of only one \emph{consistency} measure involving $\aIh u$, $\abilh$ and $\alinh$; see Theorem \ref{a:th.est.var} below. As a by-product, our analysis provides a clear definition of such a consistency for a wide range of methods, including many for which this notion was never clearly highlighted. The abstract error estimate also enables us to write, in this generic setting, the well-known principle in finite differences:
\begin{equation}\label{cons+stab=conv}
\mbox{stability + consistency $\Longrightarrow$ convergence.}
\end{equation}
In Theorem \ref{th:abstract.weak} below we also establish, under a consistency assumption of the solution to the continuous dual problem, an estimate in a weaker norm than that of $V$, which mimics classical improved error estimates (e.g., in $L^2$ norm when the energy space of the problem is $H^1$) for Finite Elements, DG, etc.
\smallskip

The abstract analysis framework is then used to derive error estimates for a variety of methods for the discretisation of a variable diffusion problem.
  The first application is to conforming and non-conforming VEM, for which we derive an energy error estimate; see Theorem \ref{th:vem.energy} below. This estimate is similar to \cite[Theorem 6.2]{Cangiani.Manzini.ea:17}, but two additional features deserve to be highlighted: it is obtained as a consequence of a general abstract framework, and the dependencies with respect to the diffusion tensor are accurately tracked. In particular, this estimate reveals that the multiplicative constant in the right-hand side is independent of the heterogeneity of the diffusion field, but depends on the square root of the (local) anisotropy ratio, a behaviour already documented for Hybrid High-Order (HHO) methods; see, e.g., \cite{Di-Pietro.Ern:17}.
  A unified $L^2$ error estimate covering both conforming and non-conforming VEM is provided in Theorem \ref{th:vem.L2}.
Establishing the VEM error estimates requires optimal approximation properties of the oblique elliptic projector. These properties, that are also of interest for other high-order methods (e.g.~the HHO method), are the purpose of Section \ref{sec:dproj}; their proof relies on the classical Dupont--Scott approximation theory \cite{Dupont.Scott:80,Brenner.Scott:08}.
  
  In the second application, we consider finite volume (FV) methods, both cell-centred and cell- and face-centred. The notion of consistency for such methods has been discussed in various references (see e.g.~\cite[Section 2.1]{Eymard.Gallouet.ea:00} or \cite[Remark 1.3]{Droniou:14}), but never directly related to error estimates. We show that the abstract analysis framework yields such estimates in terms of the consistency error purely based on the fluxes. As in the case of VEM, this error estimate is established in a diffusion-dependent discrete norm, which enables us to explicitly track the local dependencies with respect to the diffusion tensor. As an important application, we obtain the first error estimate for Multi-Point Flux Approximation L and G methods.
Several papers have tackled the question of designing a uniform convergence analysis framework for finite volume element methods \cite{Mishev:99,Chou.Li:00,Ewing.Lazarov.ea:00,Chatzipantelidis:02}, which are specific forms of finite volume methods on triangles/tetrahedra obtained by writing a balance of fluxes of a conforming or non-conforming $\mathbb{P}_1$ finite element function over a dual mesh. In these references, error estimates are obtained by writing these methods under a Petrov--Galerkin formulation \eqref{a:pro.h}. These estimates do not come from consistency errors of the fluxes but, in \cite{Chatzipantelidis:02} for example, from consistency errors involving $\abil-\abilh$ and $\alin-\alinh$; additionally, they are obtained under a global Lipschitz assumption on the diffusion tensor, and do not track the local dependency on its anisotropy ratio. Estimates in terms of flux consistencies (as in Theorems \ref{th:fv.ener.est} and \ref{th:fv.energy.linex} below) seem natural in the FV setting, and enable us to encompass all finite volume methods, including important practical ones such as MPFA schemes (that are not finite volume element methods).

  \begin{taggedblock}{dg}
    A third application, to DG methods, is considered in Appendix \ref{app:dg}.
    In this context, we show that known error estimates can be recovered with simplified arguments: first, no additional regularity is required on the exact solution in order to express consistency; second, we have a clear notion of the primal-dual consistency required for optimal $L^2$ error estimates.
  \end{taggedblock}

The rest of the paper is organised as follows. In Section \ref{sec:abstract.analysis} we present the abstract analysis framework. The main error estimates are stated in Theorem \ref{a:th.est.var} (energy norm) and \ref{th:abstract.weak} (weaker norm).
Applications of the abstract analysis framework to VEM and FV methods are considered in Section \ref{sec:applications}.
Finally, some conclusions are drawn in Section \ref{sec:conclusion}.


\section{Abstract analysis framework}\label{sec:abstract.analysis}

\subsection{Setting}\label{sec:abstract.setting}

We consider here a setting where the continuous and discrete problems are both written under
variational formulations. For the continuous problem, we take
\begin{itemize}
\item A Hilbert space $\aH$,
\item A continuous bilinear form $\abil:\aH\times\aH\to \Real$,
\item A continuous linear form $\alin:\aH\to \Real$.
\end{itemize}
The problem we aim at approximating is
\begin{equation}\label{a:pro}
\mbox{Find $u\in\aH$ such that}\quad \abil(u,v)=\alin(v)\quad\forall v\in\aH.
\end{equation}
In what follows, problem \eqref{a:pro} is named the \emph{continuous problem} in reference to the fact that the space $\aH$ is usually infinite dimensional.

\begin{remark}[Existence of a continuous solution]
  We assume the existence of a solution to problem \eqref{a:pro}; this existence follows for example from the
Lax--Milgram or Babu\v{s}ka--Brezzi lemmas if $\abil$ is  coercive or satisfies an inf--sup
condition.
\end{remark}

Our approximation is written in fully discrete Petrov--Galerkin form, using trial and test spaces that are possibly different from each other, and not necessarily spaces of functions. In particular, they are not necessarily embedded in any natural space in which $\aH$ is also embedded. We consider thus
\begin{itemize}
\item Two vector spaces $\aXh$ and $\aYh$, with respective norms $\norm[\aXh]{{\cdot}}$ and $\norm[\aYh]{{\cdot}}$.
\item A bilinear form $\abilh:\aXh\times\aYh\to\Real$.
\item A linear form $\alinh:\aYh\to \Real$.
\end{itemize}

\begin{remark}[Discrete spaces]
The spaces $\aXh$ and $\aYh$ are always finite-dimensional in applications, but this is not required in our
analysis.
The index $h$ represents a discretisation parameter (e.g., the meshsize) which characterises these spaces, and such that convergence of the method (in a sense to be made precise) is expected when $h\to 0$.
Likewise, the continuity of $\abilh$ or $\alinh$ is not directly used, but
is always verified in practice, and of course usually required to ensure the existence of a solution.
\end{remark}

The approximation of problem \eqref{a:pro} is
\begin{equation}\label{a:pro.h}
\mbox{Find $u_h\in\aXh$ such that}\quad \abilh(u_h,v_h)=\alinh(v_h)\quad\forall v_h\in\aYh.
\end{equation}
In what follows, \eqref{a:pro.h} is named the \emph{discrete problem}, in reference to the fact that the spaces $\aXh$ and $\aYh$ are usually finite dimensional.
We intend to compare the solutions to \eqref{a:pro} and \eqref{a:pro.h}
by estimating $u_h-\aIh u$, where $\aIh u$ is an element of $\aXh$ representative of the solution
$u$ to \eqref{a:pro}; see Remark \ref{rem:aIh.u}.

\begin{remark}[Equivalent Galerkin formulation]
  When the spaces $\aXh$ and $\aYh$ are finite-dimensional, their dimensions must coincide in order for the discrete problem \eqref{a:pro.h} to be well-posed.
  In this case, there exists an isomorphism $\mathfrak{I}_h:\aXh\to\aYh$, and an equivalent Galerkin formulation can be written based on the linear and bilinear forms $\widetilde{\ell}_h:\aXh\to\Real$ and $\widetilde{a}_h:\aXh\times\aXh\to\Real$ such that $\widetilde{\ell}_h(v_h)=\alinh(\mathfrak{I}_hv_h)$ and $\widetilde{a}_h(u_h,v_h)=\abilh(u_h,\mathfrak{I}_hv_h)$ for all $u_h,v_h\in \aXh$.
\end{remark}

\subsection{Error estimate in energy norm}

We now describe a notion of stability of $\abilh$ that yields a bound on the solutions
to \eqref{a:pro.h}

\begin{definition}[Inf--sup stability]\label{def:infsup.stable}
The bilinear form $\abilh$ is inf--sup stable for $(\norm[\aXh]{{\cdot}},\norm[\aYh]{{\cdot}})$ if
\begin{equation}\label{a:infsup}
  \exists \coer>0\mbox{ such that }
  \coer\norm[\aXh]{u_h}
  \le\sup_{v_h\in\aYh\backslash\{0\}}\frac{\abilh(u_h,v_h)}{\norm[\aYh]{v_h}}\quad\forall u_h\in\aXh.
\end{equation}
\end{definition}

\begin{remark}[Uniform inf--sup stability]\label{arem:unif.infsup}
In practice, one typically requires that the real number $\coer$ is independent of discretization parameters such as the meshsize. Hence, condition \eqref{a:infsup} should be verified uniformly with respect to $h$. This is needed to have optimal error estimates.
\end{remark}

\begin{remark}[Coercivity]\label{rem:coer} The inf--sup stability is of course satisfied if $\aXh=\aYh$ and $\abilh$ is coercive
in the sense that $\abilh(v_h,v_h)\ge \coer \norm[\aXh]{v_h}^2$ for all $v_h\in \aXh$, where $\coer$
does not depend on $v_h$.
\end{remark}
We next prove an a priori bound on the discrete solution. To this end, we recall that, if $Z$ is a Banach space with norm $\norm[Z]{{\cdot}}$, the dual norm of a linear form $\mu:Z\to\Real$ is classically defined by
\begin{equation}\label{a:dual.norm}
\norm[Z^\star]{\mu}=\sup_{z\in Z\backslash\{0\}}\frac{|\mu(z)|}{\norm[Z]{z}}.
\end{equation}

\begin{proposition}[A priori bound on the discrete solution]\label{prop:stab.var}
If $\abilh$ is inf--sup stable in the sense of Definition \ref{def:infsup.stable},
$m_h:\aYh\to\Real$ is linear, and $w_h$ satisfies
$$
\abilh(w_h,v_h)= m_h(v_h)\quad\forall v_h\in\aYh,
$$
then
\[
\norm[\aXh]{w_h}\le \coer^{-1}\norm[\aYh^\star]{m_h}.
\]
\end{proposition}

\begin{proof}
Take $v_h\in \aYh\backslash\{0\}$ and write, by definition of $\norm[\aYh^\star]{{\cdot}}$,
\[
\frac{\abilh(w_h,v_h)}{\norm[\aYh]{v_h}}=\frac{m_h(v_h)}{\norm[\aYh]{v_h}}\le \norm[\aYh^\star]{m_h}.
\]
The proof is completed by taking the supremum over such $v_h$ and using \eqref{a:infsup}.
\end{proof}

We then define the key notion of consistency which, in combination with the inf--sup stability,
provides the estimate on $u_h-\aIh u$ in the $\aXh$ norm.

\begin{definition}[Consistency error and consistency]\label{a:cons.error}
Let $u$ be the solution to the continuous problem \eqref{a:pro} and take
$\aIh u\in\aXh$.
The \emph{variational consistency error} is the linear form $\Cerr{u}{\cdot}:\aYh\to \Real$ defined by
\begin{equation}\label{var.cons}
\Cerr{u}{\cdot}=\alinh(\cdot) - \abilh(\aIh u,\cdot).
\end{equation}
Let now a family $(\aXh,\abilh,\alinh)_{h\to 0}$ of spaces and forms be given, and consider the corresponding family of discrete problems \eqref{a:pro.h}.
We say that \emph{consistency} holds if
$$
\text{$\norm[\aYh^\star]{\Cerr{u}{\cdot}}\to 0$ as $h\to 0$.}
$$
\end{definition}

\begin{remark}[Choice of $\aIh u$]\label{rem:aIh.u}
No particular property is required here on $\aIh u$; it could actually
be any element of $\aXh$. However, for the estimates that follow to be meaningful, it is
expected that $\aIh u$ is computed from $u$, not necessarily in a linear way but such
that information on $\aIh u$ encodes meaningful information on $u$ itself.
\end{remark}

The first main result of the paper, an estimate on $\norm[\aXh]{u_h-\aIh u}$, is stated in the following theorem.
As explained in the introduction, this theorem can be considered as a ``third Strang lemma''.
In passing, it also shows that \eqref{cons+stab=conv} holds.

\begin{theorem}[Abstract error estimate and convergence in energy norm]\label{a:th.est.var}
Assume that $\abilh$ is inf--sup stable in the sense of Definition \ref{def:infsup.stable}. Let $u$ be a solution to \eqref{a:pro}, $\aIh u\in \aXh$, and
recall the definition \eqref{var.cons} of the variational consistency error $\Cerr{u}{\cdot}$.
If $u_h$ is a solution to \eqref{a:pro.h} then
\begin{equation}\label{energy.est}
\norm[\aXh]{u_h-\aIh u}\le \coer^{-1}\norm[\aYh^\star]{\Cerr{u}{\cdot}}.
\end{equation}
As a consequence, letting a family $(\aXh,\abilh,\alinh)_{h\to 0}$ of spaces and forms be given, if consistency holds and $\coer$ does not depend on $h$, then we have \emph{convergence} in the following sense:
$$
\text{$\norm[\aXh]{u_h-\aIh u}\to 0$ as $h\to 0$.}
$$
\end{theorem}

\begin{proof}
For any $v_h\in\aYh$, the scheme \eqref{a:pro.h} yields
\begin{equation*}
\abilh(u_h-\aIh u,v_h)=\abilh(u_h,v_h)-\abilh(\aIh u,v_h)
	=\alinh(v_h)-\abilh(\aIh u,v_h).
\end{equation*}
Recalling the definition of the consistency error, we then infer that the error $u_h-\aIh u$ can be characterised as the solution to the following \emph{error equation}:
\begin{equation}
\abilh(u_h-\aIh u,v_h)=\Cerr{u}{v_h}\qquad\forall v_h\in \aYh.
\label{a:error.equation}
\end{equation}
The proof is completed by applying Proposition \ref{prop:stab.var} to $m_h=\Cerr{u}{\cdot}$ and $w_h=u_h-\aIh u$.
\end{proof}

\begin{remark}[Quasi-optimality of the error estimate]\label{rem:optimal.error.estimate}
Let 
\[
\norm[\aXh\times\aYh]{\abilh}
\coloneq\sup_{w_h\in X_h\setminus\{0\},v_h\in Y_h\setminus\{0\}}\frac{|\abilh(w_h,v_h)|}{\norm[\aXh]{w_h}\norm[\aYh]{v_h}}
\]
be the standard norm of the bilinear form $\abilh$. The error equation \eqref{a:error.equation} shows that
\[
\norm[\aYh^\star]{\Cerr{u}{\cdot}}\le \norm[\aXh\times\aYh]{\abilh} \norm[\aXh]{u_h-\aIh u}.
\]
Hence, if $\norm[\aXh\times\aYh]{\abilh}$ (and $\coer$, see Remark \ref{arem:unif.infsup})
remains bounded with respect to $h$ as $h\to 0$, which is always the case in practice,
the estimate \eqref{energy.est} is quasi-optimal in the sense that, for some
$C$ not depending on $h$, it holds that
\[
C^{-1}\norm[\aYh^\star]{\Cerr{u}{\cdot}}\le \norm[\aXh]{u_h-\aIh u}\le 
C\norm[\aYh^\star]{\Cerr{u}{\cdot}}.
\]
\end{remark}

\subsection{Improved error estimate in a weaker norm}\label{sec:est.L}

Assume now that $\aH$ is continuously embedded in a Banach space $\aL$, with norm
denoted by $\norm[\aL]{\cdot}$, and that there exists a linear reconstruction operator
\begin{equation}\label{a:recons}
  r_h:\aXh\to \aL.
\end{equation}
If $r_h$ is continuous, with norm bounded above by $C$, then \eqref{energy.est} readily gives
\begin{equation}\label{basic.rh}
\norm[\aL]{r_h(u_h-\aIh u)}\le C \coer^{-1}\norm[\aYh^\star]{\Cerr{u}{\cdot}}.
\end{equation}
Our aim here is to improve this estimate by using an Aubin--Nitsche trick.
To this purpose, we assume that, for all $g\in \aL^\star$ (the space of continuous linear forms $\aL\to \Real$), 
there exists a solution to the continuous dual problem:
\begin{equation}\label{a:pro.dual}
\mbox{Find $z_g\in \aH$ such that}\quad\abil(w,z_g)=g(w)\quad\forall w\in\aH.
\end{equation}

\begin{definition}[Dual consistency error]\label{a.def:dual.cons}
Under Assumption \eqref{a:recons}, let $g\in\aL^\star$,
$z_g$ be a solution to the dual problem \eqref{a:pro.dual}, and $\aJh z_g\in\aYh$.
The dual consistency error of $z_g$ is the linear form $\Cerrdual{z_g}{\cdot}:\aXh \to\Real$ defined by
\begin{equation}\label{var.cons.dual}
	\Cerrdual{z_g}{\cdot}=g\circ r_h - \abilh(\cdot,\aJh z_g).
\end{equation}
\end{definition}

\begin{theorem}[Improved estimate in $\aL$-norm]\label{th:abstract.weak}
Assume \eqref{a:recons} and that the dual problem \eqref{a:pro.dual} has a solution $z_g$ for any $g\in\aL^\star$.
Let $B_{\aL^\star}=\{g\in\aL^{\star}\st\norm[\aL^\star]{g}\le 1\}$ be the unit ball
in $\aL^\star$. Let $u$ and $u_h$ be the solutions to \eqref{a:pro} and \eqref{a:pro.h},
respectively, and take $\aIh u\in \aXh$ and, for $g\in B_{\aL^\star}$, $\aJh z_g\in\aYh$.
Then,
\begin{equation}\label{a:weak.est}
    \norm[\aL]{r_h(u_h-\aIh u)}\\
		\le\norm[\aXh]{u_h-\aIh u}\sup_{g\in B_{\aL^\star}}\norm[\aXh^\star]{\Cerrdual{z_g}{\cdot}}
    {+} {\sup_{g\in B_{\aL^\star}}}\Cerr{u}{\aJh z_g}.
\end{equation}
\end{theorem}

\begin{remark}[Primal-dual consistency error]
The quantity $\Cerr{u}{\aJh z_g}=\alinh(\aJh z_g)-\abilh(\aIh u,\aJh z_g)$ is a measure of consistency
of the discrete primal problem \eqref{a:pro.h} that also involves the solution $z_g$ to the continuous dual problem \eqref{a:pro.dual}. For this reason, we will call $\Cerr{u}{\aJh z_g}$ the \emph{primal-dual consistency error}.
\end{remark}

\begin{proof} Let $g\in B_{\aL^\star}$.
  By definition \eqref{var.cons.dual} of $\Cerrdual{z_g}{\cdot}$, it holds for any $w_h\in\aXh$,
\[
	g(r_hw_h)=\Cerrdual{z_g}{w_h}+\abilh(w_h,\aJh z_g).
\]
Letting $w_h=u_h-\aIh u$ and recalling the error equation \eqref{a:error.equation},
this gives
\[
	g(r_h(u_h-\aIh u))=\Cerrdual{z_g}{u_h-\aIh u}+\Cerr{u}{\aJh z_g}.
\]
Taking the supremum over $g\in B_{\aL^\star}$, and recalling that
$\sup_{g\in B_{\aL^\star}}g(w)=\norm[\aL]{w}$ for all $w\in\aL$, we infer
\begin{equation}\label{alternative.est.L}
	\norm[\aL]{r_h(u_h-\aIh u)}\le \sup_{g\in B_{\aL^\star}}\Cerrdual{z_g}{u_h-\aIh u}+
 \sup_{g\in B_{\aL^\star}}\Cerr{u}{\aJh z_g}.
\end{equation}
To conclude, recall the definition \eqref{a:dual.norm} of the dual norm to write
$$
\Cerrdual{z_g}{u_h-\aIh u}\le \norm[\aXh]{u_h-\aIh u}\norm[\aXh^\star]{\Cerrdual{z_g}{\cdot}}.\qedhere
$$
\end{proof}

\begin{remark}[Alternative $\aL$-error bound]
The estimate \eqref{alternative.est.L} appears slightly sharper than \eqref{a:weak.est}.
In the statement of Theorem \ref{th:abstract.weak}, however, we have preferred a formulation which emphasises a general property of the dual consistency error $\Cerrdual{z_g}{\cdot}$ rather than its evaluation at a specific argument; indeed, unlike $\Cerr{u}{\aJh z_g}$ (see Section \ref{sec:rates}), it does not seem possible in general to have a better bound on $\Cerrdual{z_g}{u_h-\aIh u}$ than the one provided by $\norm[\aXh]{u_h-\aIh u}\norm[\aXh^\star]{\Cerrdual{z_g}{\cdot}}$.
\end{remark}

\subsection{Comments}

A few comments are in order.

\subsubsection{Recovering continuous estimates}

For a number of methods, the interpolation operator $\aIh$ is naturally defined as part of the method, and there exists some continuous linear reconstruction operator $R_h:\aXh\to E$, where $E$ is a space of functions (which might or might not be a subspace of $H$).

For example, in conforming FE methods, $\aIh$ is the nodal interpolant and $R_h$ the reconstruction of functions in the FE space from their nodal values. In HHO or non-conforming VEM methods, $\aIh$ corresponds to $L^2$ projections on local (face- and cell-) polynomial spaces, and $R_h$ is a local potential reconstruction related to the elliptic projector; alternatively, in the VEM setting, $R_h v_h$ can give the unique (but not explicitly known) function in the VEM space that has the degrees of freedom encoded in the vector $v_h$.

If the norm of $R_h$ is bounded by $C$, the energy estimate \eqref{energy.est} gives
\[
\norm[E]{R_h u_h - R_h \aIh u}\le C\coer^{-1}\norm[\aYh^\star]{\Cerr{u}{\cdot}}.
\]
The triangle inequality then leads to the following continuous estimate between the reconstructed function $R_h u_h$ and the solution $u$ to the continuous problem:
\[
\norm[E]{R_h u_h - u}\le C\coer^{-1}\norm[\aYh^\star]{\Cerr{u}{\cdot}}+\norm[E]{R_h \aIh u_h-u}.
\]
The last term is usually estimated by means of approximation properties of the space $\aXh$ and of the operators $\aIh$ and $R_h$ attached to the scheme (they do not depend on the continuous equation \eqref{a:pro} or its discretisation \eqref{a:pro.h}). Hence, even though Theorem \ref{a:th.est.var} states an estimate in a purely discrete setting, from this a continuous estimate can often be easily recovered.

\subsubsection{Link between primal, dual and primal-dual consistency errors}\label{sec:link.cons}

If the discrete bilinear form $\abilh$ is symmetric (which requires $\aXh=\aYh$) and $\alinh=\ell\circ r_h$, then the primal and dual consistency errors are identical, and thus estimating $\norm[\aXh^\star]{\Cerrdual{z_g}{\cdot}}$ in \eqref{a:weak.est} does not require any additional work than the one done for estimating $\norm[\aXh^\star]{\Cerr{u}{\cdot}}$ in \eqref{energy.est}.

Even if $\abilh$ is not symmetric or $\alinh\not=\ell\circ r_h$, the dual problem \eqref{a:pro.dual} often has a similar structure as the primal problem, with different parameters; this is expected to be reflected in $\Cerrdual{z_g}{\cdot}$, which might simply be $\Cerr{z_g}{\cdot}$ with different parameters. In this case, the estimate done on the primal consistency error might directly apply, with easy substitutions, to the dual consistency error. For example, the dual problem to the advection--diffusion--reaction model
\begin{equation}\label{pb.adr}
-\DIV(\diff\GRAD u)+\DIV(\vel u)+\reac u=f\mbox{ in $\Omega$,}\qquad u=0\mbox{ on $\partial\Omega$},
\end{equation}
is the same problem with $\vel$ replaced with $-\vel$ and $\reac$ replaced with $\reac+\DIV\vel$.

Estimating $\norm[\aYh^\star]{\Cerr{u}{\cdot}}$ in \eqref{energy.est} requires to estimate $\Cerr{u}{v_h}$ for all $v_h\in \aYh$. Some of the steps performed in this estimate can often be directly used to estimate the primal-dual consistency error $\Cerr{u}{\aJh z}$ in \eqref{a:weak.est}. One simply needs to be cautious and draw on the additional information available in this latter term: the primal consistency error is not tested on an arbitrary $v_h\in \aYh$, but on the specific vector $\aJh z_g$; taking advantage of that specificity can lead to improved rates of convergence (see next section). This idea is illustrated in the proof of Theorem \ref{th:vem.L2} below.

We also notice that, when multiple terms are present as in problem \eqref{pb.adr}, the consistency error involves one component per term, and these components  can be estimated independently.
The benefit is twofold: on the one hand, proceeding this way simplifies the analysis; on the other hand, it makes it possible to re-use the consistency results proved individually for each operator.

\subsubsection{Integer and fractional rates of convergence} \label{sec:rates}

Under some regularity assumptions on the solution $u$, it is possible to obtain rates of convergence for the quantity $\norm[\aYh^\star]{\Cerr{u}{\cdot}}$ that appears in the right-hand side of the energy error estimate \eqref{energy.est}. Typically, for second order elliptic problems, one will assume that $u\in H^r(\Omega)\cap H^1_0(\Omega)$ and establish that
\begin{equation}\label{est:Hr}
\norm[\aYh^\star]{\Cerr{u}{\cdot}}\le C h^{\omega(r)}\norm[H^r(\Omega)]{u},
\end{equation}
where $\omega(r)$ is an appropriate power depending on $r$ and on the considered scheme. This estimate is usually easier to establish for integer $r$, but once this is done it also holds for fractional $r$, by basic interpolation result on the mapping $u\mapsto \Cerr{u}{\cdot}\in\aYh^\star$. The same considerations holds for the dual consistency error.

Let us now examine the improved $L$-error estimate \eqref{a:weak.est} and consider, for example, elliptic problems. Under optimal elliptic regularity assumptions, it is expected that $z_g\in H^2(\Omega)$. This regularity result will translate into a specific rate of convergence of $\norm[\aXh^\star]{\Cerrdual{z_g}{\cdot}}$, say $\mathcal O(h)$. The first term in \eqref{a:weak.est} is then one (or more) orders of magnitude less that $\norm[\aXh]{u_h-\aIh u}$. The regularity of $z_g$ also translates into constraints on the vector $\aJh z_g$, which cannot vary as freely as any $v_h\in \aYh$; because of that, it is expected that the primal-dual consistency error $\Cerr{u}{\aJh z_g}$ is also one or more orders of magnitude less that $\norm[\aYh^\star]{\Cerr{u}{\cdot}}$. Hence, the right-hand side of \eqref{a:weak.est} should converge at a faster rate than the right-hand side of \eqref{energy.est} as $h\to 0$, showing that Theorem \ref{th:abstract.weak} is indeed an improvement over the basic estimate \eqref{basic.rh} coming from Theorem \ref{a:th.est.var}. 

\subsubsection{Range of applications}

Let us explicitly remark that, even though we only consider, for questions of length, second order elliptic problems in Section \ref{sec:applications}, the framework and estimates described in this section cover a wide range of equations and schemes. For example, elliptic equations of order four, such as the ones encountered in the theory of thin plates, also fit into the setting of Section \ref{sec:abstract.setting}. 
Several popular numerical tricks are also covered by the present framework, such as the weak enforcement (\emph{\`a la} Nitsche) of boundary conditions in the discrete formulation \eqref{a:pro.h}.

Finally, we also note that, even though this is the classical example we might have in mind for second order elliptic problems, the space $\aL$ in Section \eqref{sec:est.L} does not need to be $L^2(\Omega)$. It could for example be $H^s(\Omega)$ for some $s\in (0,1)$, leading to optimal rates of convergence in $H^s$ norm instead of $L^2$ norm.


\section{Applications}\label{sec:applications}

In this section we showcase applications of the discrete analysis framework to a variety of numerical methods.

\subsection{Setting}\label{model:setting}

For the sake of simplicity, we focus on a pure diffusion model problem.
Denote by $\Omega\subset\Real^d$, $d\ge 1$, an open bounded connected polytopal domain with boundary $\partial\Omega$.
In what follows, for any measured set $X$, we denote by $({\cdot},{\cdot})_X$ the usual inner product of $L^2(X)$ or $L^2(X)^d$ according to the context, by $\|{\cdot}\|_X$ the corresponding norm, and we adopt the convention that the subscript is omitted whenever $X=\Omega$.

Let $\diff:\Omega\to\Real^{d\times d}$ denote a symmetric, uniformly elliptic diffusion field, which we additionally assume piecewise constant on a finite partition $P_\Omega=\{\Omega_i\st 1\le i\le N_\Omega\}$ of $\Omega$ into polytopes.
For a given source term $f:\Omega\to\Real$, our model problem reads:
Find $u:\Omega\to\Real$ such that
\begin{equation}\label{eq:strong}
	-\DIV(\diff\GRAD u)=f\quad\mbox{in $\Omega$},\qquad u=0\quad\mbox{on $\partial\Omega$}.
\end{equation}
Assuming $f\in L^2(\Omega)$, a weak formulation of this problem is: Find $u\in H_0^1(\Omega)$ such that
\begin{equation}\label{eq:weak}
  a_{\diff}(u,v)\coloneq(\diff\GRAD u,\GRAD v)=(f,v)\qquad\forall v\in H_0^1(\Omega).
\end{equation}

We denote by $\Mh=(\Th,\Fh)$ a mesh of the domain, where $\Th$ collects the mesh elements, or cells, and $\Fh=\Fhi\cup\Fhb$ the hyperplanar mesh faces, with $\Fhi$ and $\Fhb$ denoting, respectively, the sets of internal and boundary faces.
For any mesh element $T\in\Th$, $h_T$ is the diameter of $T$ and $\Fh[T]$ is the set of faces that lie on its boundary $\partial T$.
Symmetrically, for any mesh face $F\in\Fh$, we denote by $h_F$ the diameter of $F$ and by $\Th[F]$ the set collecting the one (if $F$ is a boundary face) or two (if $F$ is an internal face) mesh elements that share $F$.
For any $T\in\Th$ and any $F\in\Fh[T]$, $\normal_{TF}$ is the unit vector normal to $F$ and pointing out of $T$.

The meshes we consider are always part of a regular family $(\Mh)_{h\in\mathcal H}$ in the sense of \cite[Definition 3.3]{Di-Pietro.Tittarelli:17}. Unless otherwise specified, the notation $a\lesssim b$ means $a\le Cb$ with constant $C$ possibly depending on the regularity factor of that family, but not depending on $\diff$ or $h$ and, for local inequalities, on the mesh element or face. Additional regularity assumptions on the meshes depend on the considered method and will be given when necessary.
We however always assume that the mesh is compliant with the partition $P_\Omega$, i.e., for all $T\in\Th$, there exists a unique $\Omega_i$, $1\le i\le N_\Omega$, such that $T\subset\Omega_i$.
For all $T\in\Th$, we denote by $\diff[T]\coloneq\diff[|T]$ the constant value of $\diff$ inside $T$ and we introduce the local anisotropy ratio
\begin{equation}\label{def:alphaT}
\alpha_T\coloneq\frac{\overline{\lambda}_T}{\underline{\lambda}_T}
\end{equation}
where $\underline{\lambda}_T$ and $\overline{\lambda}_T$ denote, respectively, the smallest and largest eigenvalues of $\diff[T]$.

For a given integer $\polydeg\ge 0$ and $X$ mesh element or face, we denote by $\Poly{\polydeg}(X)$ the space spanned by the restriction to $X$ of $d$-variate polynomials of total degree $\le\polydeg$.
The $L^2$-projector $\lproj[X]{\polydeg}:L^2(X)\to \Poly{\polydeg}(X)$ is defined by: For all $w\in L^2(X)$, $\lproj[X]{\polydeg}w$ is the unique element in $\Poly{\polydeg}(X)$ such that $(\lproj[X]{\polydeg}w,q)_X=(w,q)_X$ for all $q\in\Poly{\polydeg}(X)$.
In the discussion, we will need the following approximation results, which are a special case of \cite[Lemmas 3.4 and 3.6]{Di-Pietro.Droniou:17}: Let an integer $s\in\{0,\ldots,\polydeg+1\}$ be given. Then, for any mesh element $T\in\Th$, any function $v\in H^s(T)$, and any exponent $m\in\{0,\ldots,s\}$, it holds that
\begin{equation}\label{eq:lproj:approx}
  \seminorm[H^m(T)]{v-\lproj[T]{\polydeg}v}\lesssim h_T^{s-m}\seminorm[H^s(T)]{v}.
\end{equation}
Moreover, if $s\ge 1$ and $m\le s-1$,
\begin{equation}\label{eq:lproj:approx.trace}
  \seminorm[{H^m(\Fh[T])}]{v-\lproj[T]{\polydeg}v}\lesssim h_T^{s-m-\frac12}\seminorm[H^s(T)]{v},
\end{equation}
where $H^m(\Fh[T])\coloneq\left\{v\in L^2(\partial T)\st v_{|F}\in H^m(F)\mbox{ for all }F\in\Fh[T]\right\}$ is the broken Sobolev space on $\Fh[T]$ and $\seminorm[{H^m(\Fh[T])}]{{\cdot}}$ the corresponding broken seminorm.

The space of broken polynomials of total degree $\le\polydeg$ on $\Th$ is denoted by $\Poly{\polydeg}(\Th)$, i.e.,
$$
\Poly{\polydeg}(\Th)\coloneq\left\{v\in L^2(\Omega)\st v_{|T}\in\Poly{\polydeg}(T)\quad\forall T\in\Th\right\}.
$$
For a given exponent $s\in\Natural$, we define the broken Sobolev space
$$
H^s(\Th)\coloneq\left\{v\in L^2(\Omega): v_{|T}\in H^s(T)\quad\forall T\in\Th\right\}.
$$
On $H^1(\Th)$ we define the broken gradient operator $\GRADh$ such that, for all $v\in H^1(\Th)$, $(\GRADh v)_{|T}\coloneq\GRAD v_{|T}$.

\subsection{Virtual Element Methods}\label{sec:vem}

The first application we consider is to VEM.
The main novelty of this section is the derivation of a unified energy error estimate for both conforming and non-conforming VEM where the dependence on the diffusion field is accurately tracked.
Our results show full robustness with respect to the heterogeneity of the diffusion field, and a mild dependence on the square root of the local anisotropy ratio.

\subsubsection{The oblique elliptic projector}\label{sec:dproj}

Like several other arbitrary-order discretisation methods for problem \eqref{eq:weak} (such as HHO \cite{Di-Pietro.Ern.ea:14}, Weak Galerkin \cite{Wang.Ye:13} methods, and Mimetic Finite Differences \cite{Lipnikov.Manzini:14}), VEM are based on local projectors which possibly embed a dependence on the diffusion field inside $T$. Let $k\ge 1$ be a natural number. Fixing $T\in\Th$, we focus here on VEM formulations based on the (oblique) elliptic projector $\dproj{k}:H^1(T)\to \Poly{k}(T)$  defined by: For $v\in H^1(T)$,
\begin{subequations}\label{def:dproj}
\begin{align}
\label{def:dproj.1}
(\diff[T]\GRAD\dproj{k}v,\GRAD w)_T={}&(\diff[T]\GRAD v,\GRAD w)_T\quad\forall w\in\Poly{k}(T),\\
\label{def:dproj.2}
\int_T \dproj{k}v={}&\int_T v.
\end{align}
\end{subequations}
By the Riesz representation theorem in $\GRAD\Poly{k}(T)$, \eqref{def:dproj.1} defines a unique
element of $\GRAD \Poly{k}(T)$, and the closure equation \eqref{def:dproj.2}
fixes the corresponding $\dproj{k}v\in\Poly{k}(T)$.
It can be easily checked that $\dproj{k}$ is a projector, i.e., it is linear and idempotent.
As a result, it maps polynomials of total degree $\le k$ onto themselves.
Optimal approximation properties for $\dproj{k}$ in diffusion-dependent seminorms are studied in the following theorem, where the dependence of the multiplicative constants on the local diffusion tensor $\diff[T]$ is carefully tracked.
\begin{theorem}[Approximation properties of the oblique elliptic projector in diffusion-weighted seminorms]\label{thm:approx.dproj}
  Assume the setting described in Section \ref{model:setting}.
  For a given polynomial degree $k\ge0$, let an integer $s\in\{1,\ldots,k+1\}$ be given.
  Then, recalling the definition \eqref{def:dproj} of the oblique elliptic projector, for all $v\in H^{s}(T)$ and all $m\in \{0,\ldots,s-1\}$,
  \begin{equation}\label{dproj:approx}
    \seminorm[H^m(T)^d]{\diff[T]^{\frac12}\GRAD(v-\dproj{k}v)}
    \lesssim\overline{\lambda}_T^{\frac12} h_T^{s-m-1}\seminorm[H^{s}(T)]{v}.
  \end{equation}
  If, additionally, $m\le s-2$ (which enforces $s\ge 2$), then
  \begin{equation}\label{dproj:trace}
    h_T^{\frac12}\seminorm[{H^m(\Fh[T])^d}]{\diff[T]^{\frac12}\GRAD(v-\dproj{k}v)}
    \lesssim\overline{\lambda}_T^{\frac12} h_T^{s-m-1}\seminorm[H^{s}(T)]{v},
  \end{equation}
  where $H^m(\Fh[T])^d$ is the broken Sobolev space on $\Fh[T]$ defined component-wise as in \eqref{eq:lproj:approx.trace}, and
  $\seminorm[{H^m(\Fh[T])^d}]{{\cdot}}$ is the corresponding seminorm.
\end{theorem}
\begin{proof}
  We consider the following representation of $v$:
  \begin{equation}\label{dproj:approx:1}
    v=Q^{s} v + R^{s} v,
  \end{equation}
  where $Q^{s} v\in\Poly{s-1}(T)\subset\Poly{k}(T)$ is the averaged Taylor polynomial, while the remainder $R^{s} v$ satisfies, for all $r\in\{0,\ldots,s\}$ (cf. \cite[Lemma 4.3.8]{Brenner.Scott:08}),
  \begin{equation}\label{dproj:approx:2}
    \seminorm[H^r(T)]{R^{s} v}\lesssim h_T^{s-r}\seminorm[H^{s}(T)]{v}.
  \end{equation}
  We next notice that, using the definition \eqref{def:dproj} of the oblique elliptic projector, it holds for any $\phi\in H^1(T)$,
  \begin{equation}\label{dproj:approx:3}
    \norm[T]{\diff[T]^{\frac12}\GRAD\dproj{k}\phi}\le\norm[T]{\diff[T]^{\frac12}\GRAD \phi},
  \end{equation}
  as can be inferred selecting $w=\dproj{k}\phi$ as a function test in \eqref{def:dproj.1} and using the Cauchy--Schwarz inequality.
  Taking the projection of \eqref{dproj:approx:1}, and using the fact that $\dproj{k}$ maps polynomials of total degree $\le k$ onto themselves to write $\dproj{k}Q^{s}v=Q^{s}v$, it is inferred that $\dproj{k}v=Q^{s} v + \dproj{k}(R^{s} v)$.
  Subtracting this equation from \eqref{dproj:approx:1}, we obtain $v-\dproj{k}v=R^{s}v - \dproj{k}(R^{s}v)$. Applying the operator $\diff[T]^{1/2}\GRAD$ to this expression, passing to the seminorm, and using the triangle inequality, we arrive at
  \begin{equation}\label{dproj:approx:4}
    \seminorm[H^{m}(T)^d]{\diff[T]^{\frac12}\GRAD(v-\dproj{k}v)}
    \le
    \seminorm[H^{m}(T)^d]{\diff[T]^{\frac12}\GRAD R^{s}v}
    + \seminorm[H^{m}(T)^d]{\diff[T]^{\frac12}\GRAD\dproj{k}(R^{s}v)}
    \eqcolon \term_1 + \term_2.
  \end{equation}
  For the first term, it is readily inferred that
  $\term_1\lesssim\overline{\lambda}_T^{\frac12}\seminorm[H^{m+1}(T)]{R^{s}v}$
  which, combined with \eqref{dproj:approx:2} for $r=m+1$, gives
  \begin{equation}\label{dproj:approx:term1}
    \term_1\lesssim\overline{\lambda}_T^{\frac12} h_T^{s-m-1}\seminorm[H^{s}(T)]{v}.
  \end{equation}
  For the second term, on the other hand, we can proceed as follows:
  \begin{equation}\label{dproj:approx:term2}
    \term_2
    \lesssim h_T^{-m}\norm[T]{\diff[T]^{\frac12}\GRAD(\dproj{k}R^{s}v)}
    \lesssim h_T^{-m}\norm[T]{\diff[T]^{\frac12}\GRAD R^{s}v}
    \lesssim \overline{\lambda}_T^{\frac12} h_T^{-m}\seminorm[H^1(T)]{R^{s}v}
    \lesssim \overline{\lambda}_T^{\frac12} h_T^{s-m-1}\seminorm[H^{s}(T)]{v},
  \end{equation}
  where we have used the local inverse Sobolev embeddings of \cite[Remark A.2]{Di-Pietro.Droniou:17} in the first bound,
  \eqref{dproj:approx:3} with $\phi=R^{s}v$ in the second bound,
  the definition of the $H^1$-seminorm in the third bound,
  and \eqref{dproj:approx:2} with $r=1$ to conclude.
  Plugging \eqref{dproj:approx:term1} and \eqref{dproj:approx:term2} into \eqref{dproj:approx:4}, \eqref{dproj:approx} follows.
  To prove \eqref{dproj:trace}, it suffices to combine \eqref{dproj:approx} with a local continuous trace inequality (see \cite[Lemma 1.49]{Di-Pietro.Ern:12}, where a slightly different notion of face is used which, however, does not affect the final result).
\end{proof}

\begin{remark}[Case of varying {$\diff[T]$}]\label{approx.eproj:var.diff}
Theorem \ref{thm:approx.dproj} also holds for a diffusion that varies inside $T$, provided that $\diff[T]^{1/2}$ belongs to $\Poly{r}(T)^{d\times d}$ for some integer $r$. Under this assumption,
and letting $\overline{\lambda}_T$ be the maximum over $\vec{x}\in T$ of the largest eigenvalue of $\diff[T](\vec{x})$, \eqref{dproj:approx:term1} remains valid owing to \eqref{dproj:approx:2}  and inverse inequalities (that show that $\seminorm[W^{s,\infty}(T)]{\diff[T]^{1/2}}\lesssim h_T^{-s}\norm[L^\infty(T)]{\diff[T]^{1/2}}\lesssim h_T^{-s} \overline{\lambda}_T^{\frac12}$), and the local inverse Sobolev embeddings of \cite[Remark A.2]{Di-Pietro.Droniou:17} can be invoked to establish \eqref{dproj:approx:term2}.
In this case, the hidden multiplicative constants in \eqref{dproj:approx}--\eqref{dproj:trace} additionally depend on $r$.

If $\diff[T]^{1/2}$ is not in $\Poly{r}(T)^{d\times d}$ then, following the proof above, the right-hand sides of \eqref{dproj:approx} and \eqref{dproj:trace} have to be multiplied by $\norm[W^{m,\infty}(T)^{d\times d}]{\diff[T]^{1/2}}$.
\end{remark}

\subsubsection{An abstract Virtual Element Method}

To define a VEM scheme, one first needs to choose a finite dimensional subspace $V^k_h\subset H^1(\Th)$ that locally contains polynomials and satisfies some continuity requirements:
\begin{equation}\label{vem:PkVkT}
  \Poly{k}(T)\subset V^k_T
  \coloneq\left\{(v_h)_{|T}\,:\,v_h\in V^k_h\right\}\quad \forall T\in\Th,
\end{equation}
and, for all $v_h\in V^k_h$,
\begin{equation}\label{vem:cont.vh}
\lproj[F]{k-1}(v_h)_{|T}+\lproj[F]{k-1}(v_h)_{|T'}=0\quad\forall F\in\Fhi\mbox{ with }\Th[F]=\{T,T'\}.
\end{equation}
A subspace $X_h=V^k_{h,0}$ of $V^k_h$ is then considered to account for the homogeneous Dirichlet boundary conditions, and such that (at least) the following condition holds:
For all $v_h\in V_{h,0}^k$,
\begin{equation}\label{vem:bc.vh}
\lproj[F]{k-1}v_h=0\quad\forall F\in\Fhb.
\end{equation}
Different choices of spaces lead to different methods, such as conforming \cite{Beirao-da-Veiga.Brezzi.ea:14} or non-conforming \cite{Ayuso-de-Dios.Lipnikov.ea:16} VEM. Our analysis here does not require a complete description of the space $V_{h,0}^k$. We merely need the two following properties of the interpolant $\aIh:H^1_0(\Omega)\cap C(\overline{\Omega})\to V_{h,0}^k$:
\begin{enumerate}
\item[\textbf{(I1)}] \emph{Locality and boundedness.} For all $T\in\Th$, there is a linear mapping $I_T:H^1(T)\cap C(\overline{T})\to V^k_T$ such that $(\aIh w)_{|T}=I_T(w_{|T})$ for all $w\in H^1_0(\Omega)\cap C(\overline{\Omega})$,
and 
\begin{equation}\label{vem:IT.bound}
\norm[T]{\diff[T]^{\frac12}\GRAD (I_T \phi)}\lesssim \overline{\lambda}_T^{\frac12}\left(\norm[T]{\GRAD \phi}+h_T\seminorm[H^1(T)^d]{\GRAD \phi}\right)
\quad \forall \phi\in H^2(T).
\end{equation}
\item[\textbf{(I2)}] \emph{Preservation of polynomials.} For all $T\in\Th$ and $v\in\Poly{k}(T)$,
$I_T v=v$.
\end{enumerate}

\begin{remark}[Hypotheses \textbf{(I1)} and \textbf{(I2)}]
The preservation of polynomials \textbf{(I2)} is a trivial property of conforming and non-conforming VEM since both local spaces contain polynomials of degree $\le k$. 

By definition of $\overline{\lambda}_T$, Hypothesis \textbf{(I1)} only has to be established in the case $\diff[T]=\Id$. For conforming VEM, in the case of star-shaped mesh elements, this property follows from \cite[Lemma 2.23]{Brenner.Guan.ea:17}. For non-conforming VEM, \textbf{(I1)} with $\diff[T]=\Id$ is a consequence of \cite[Lemma 22]{Di-Pietro.Droniou.ea:18} and \cite[Proposition 7.1]{Di-Pietro.Droniou:17} (in passing, these estimates show that, for non-conforming VEM, \eqref{vem:IT.bound} holds without the term $h_T\seminorm[H^1(T)^d]{\GRAD \phi}$ and assuming only $\phi\in H^1(T)$).
\end{remark}

\begin{remark}[On the DOFs for VEM]
The degrees of freedom (DOFs) of the VEM, that is the unisolvent family of linear forms $(\lambda_i)_{i\in I}$ on $V_{h,0}^k$, must be chosen to enable the computation of the oblique elliptic projector \eqref{def:dproj} for functions in $V_{h,0}^k$. Given \eqref{def:dproj.2}, this means, in particular, that these DOFs should enable the computation, for all $T\in\Th$, of $\lproj[T]{0}$ on $V_{h,0}^k$. In the case $k=1$, we therefore implicitly consider enhanced VEM spaces \cite{Cangiani.Manzini.ea:17}. Otherwise, the closure equation \eqref{def:dproj.2} should be modified, see e.g.~\cite{Brenner.Guan.ea:17}.
\end{remark}

Let $V_{h,0}^k$ be a chosen VEM space, and
\[
l=0\mbox{ if $k=1$},\quad l\in\{0,1\}\mbox{ if $k=2$},\quad l=k-2\mbox{ if $k\ge 3$}.
\]
We assume that the DOFs of $V_{h,0}^k$ enable the computation of $(\lproj[T]{l})_{T\in\Th}$ on $V_{h,0}^k$ (in the cases $k=1$ or $(k,l)=(2,1)$, this supposes using enhanced spaces).
A VEM scheme for \eqref{eq:weak} is then obtained by writing \eqref{a:pro.h} with
\begin{equation}\label{vem:def.lh}
\alinh(v_h)=\sum_{T\in \Th} (f,\lproj[T]{l}v_h)_T\quad\forall v_h\in V_{h,0}^k
\end{equation}
and
\begin{equation}\label{vem:def.ah}
\begin{aligned}
&\abilh(v_h,w_h)=\sum_{T\in\Th} a_T(v_h,w_h)\quad\forall v_h,w_h\in V_{h,0}^k,\\
&\mbox{ where }\quad a_T(v_h,w_h)=(\diff[T]\GRAD\dproj{k}v_h,\GRAD\dproj{k}w_h)_T+s_T\left((I-\dproj{k})v_h,(I-\dproj{k})w_h\right).
\end{aligned}
\end{equation}
Here, $s_T$ is a symmetric positive semi-definite bilinear form on $V_T^k$ (computable from the degrees of freedom) such that
\begin{equation}\label{vem:def.sT}
(\diff[T]\GRAD w,\GRAD w)_T\lesssim s_T(w,w)\lesssim (\diff[T]\GRAD w,\GRAD w)_T
\quad\forall w\in V_T^k\mbox{ such that }\dproj{k}w=0.
\end{equation}
The definition of $\abilh$ implies that if $\abilh(v_h,v_h)=0$ then $v_h$ is constant in each cell, and \eqref{vem:cont.vh}--\eqref{vem:bc.vh} then show that $v_h=0$. Hence, $\abilh$ is symmetric positive definite on $V_{h,0}^k$. The norm considered on $\aXh=V_{h,0}^k$ is the one induced by $\abilh$, that is,
\begin{equation}\label{vem:def.norm}
\norm[\aXh]{v_h}\coloneq\sqrt{\abilh(v_h,v_h)}\quad\forall v_h\in\aXh.
\end{equation}

\subsubsection{Error estimate in energy norm}

With this setting in place, Theorem \ref{a:th.est.var} yields the following estimates.

\begin{theorem}[Energy estimates for VEM schemes]\label{th:vem.energy}
Let $1\le r\le k$ and assume that the solution $u\in H^1_0(\Omega)\cap C(\overline{\Omega})$ to \eqref{eq:weak} belongs to $H^{r+1}(\Th)$.
Let $u_h$ be the solution of the VEM scheme (that is, \eqref{a:pro.h} with the choices \eqref{vem:def.lh} and \eqref{vem:def.ah}). Then the following estimates hold:
\begin{equation}\label{vem:est.energy}
\norm[\aXh]{u_h-\aIh u}\lesssim \left(\sum_{T\in\Th} \alpha_T  \overline{\lambda}_T h_T^{2r}\seminorm[H^{r+1}(T)]{u}^2\right)^{\frac12}
\end{equation}
and
\begin{equation}\label{vem:est.recons}
\norm{\diff^{\frac12}(\GRADh\dproj[h]{k}u_h-\GRAD u)}\lesssim \left(\sum_{T\in\Th} \alpha_T  \overline{\lambda}_T h_T^{2r}\seminorm[H^{r+1}(T)]{u}^2\right)^{\frac12},
\end{equation}
where $\dproj[h]{k}$ is the patched elliptic projector such that, for all $T\in\Th$ and $w\in H^1(\Th)$,
$(\dproj[h]{k}w)_{|T}\coloneq\dproj{k}(w_{|T})$.
\end{theorem}
\begin{remark}[Diffusion varying in each cell]
If $\diff$ is not piecewise-constant in the cells, the construction of VEM methods has to be adjusted by using the $L^2$-orthogonal projection of the gradient of virtual functions; see, e.g., \cite{Beirao-da-Veiga.Brezzi.ea:16,Cangiani.Manzini.ea:17}. In the context of HHO methods (strongly related to non-conforming VEM), similar ideas have been used in  \cite{Di-Pietro.Droniou:17} for the approximation of fully non-linear models; see also the discussion in \cite[Section 4]{Di-Pietro.Droniou.ea:18}. A different approach in the context of HHO methods, valid for linear problems, consists in incorporating the variable diffusion coefficient into the local reconstruction; see \cite{Di-Pietro.Ern:15}.

It is worth noting that Theorem \ref{th:vem.energy} remains of interest even for a piecewise constant diffusion tensor.
Even though the $H^{r+1}(\Th)$ regularity of the solution cannot always be ascertained if $\diff$ is discontinuous (counter-examples to the $H^2(\Th)$ regularity can be constructed for some piecewise-constant diffusions and smooth source terms \cite{Tartar:15}), one can easily find many situations in which $\diff$ is piecewise-constant and the solution belongs to $H^{r+1}(\Th)$, situations for which Theorem \ref{th:vem.energy} yields a meaningful estimate. On the contrary, an estimate based on the $H^{r+1}(\Omega)$-norm would essentially impose a smooth $\diff$ over the entire domain -- as any discontinuity of the diffusion essentially prevents the solution from being in $H^2(\Omega)$ or more regular.
\end{remark}
\begin{proof}
By choice of the norm on $V_{h,0}^k$, the bilinear form $\abilh$ is coercive with constant $1$.
Hence, \eqref{vem:est.energy} follows from \eqref{energy.est} (with $\gamma=1$) if we estimate the norm of the
consistency error appropriately. Since $f=-\DIV(\diff\GRAD u)$ we have, for all $v_h\in V_{h,0}^k$,
\begin{align}
	\Cerr{u}{v_h}={}&\sum_{T\in\Th}(-\DIV(\diff[T]\GRAD u),\lproj[T]{l}v_h)_T-\sum_{T\in\Th}(\diff[T]\GRAD\dproj{k}\aIh u,\GRAD\dproj{k}v_h)_T\nonumber\\
	&-\sum_{T\in\Th}s_T\left((I-\dproj{k})\aIh u,(I-\dproj{k})v_h\right)\nonumber\\
={}&\sum_{T\in\Th}(-\DIV(\diff[T]\GRAD u),\lproj[T]{l}v_h)_T-\sum_{T\in\Th}(\diff[T]\GRAD\dproj{k}u,\GRAD\dproj{k}v_h)_T\nonumber\\
	&-\sum_{T\in\Th}(\diff[T]\GRAD\dproj{k}(\aIh u-u),\GRAD\dproj{k}v_h)_T-\sum_{T\in\Th}s_T\left((I-\dproj{k})\aIh u,(I-\dproj{k})v_h\right)\nonumber\\
=:{}&\term_1+\term_2+\term_3+\term_4.
\label{vem:cons.1}
\end{align}
\medskip\\
(i) \emph{Term $\term_1+\term_2$}. Performing element-wise integrations-by-parts, we write
\begin{align}
	\term_{1}={}&\sum_{T\in\Th}(\diff[T]\GRAD u,\GRAD \lproj[T]{l}v_h)_T-\sum_{T\in\Th}\sum_{F\in\Fh[T]}(\diff[T]\GRAD u\SCAL\normal_{TF},\lproj[T]{l}v_h)_F
	\nonumber\\
	={}&\sum_{T\in\Th}(\diff[T]\GRAD u,\GRAD \lproj[T]{l}v_h)_T-\sum_{T\in\Th}\sum_{F\in\Fh[T]}(\diff[T]\GRAD u\SCAL\normal_{TF},\lproj[T]{l}v_h-\lproj[F]{k-1}v_h)_F,
	\label{vem:T1}
\end{align}
where the introduction of the term $\lproj[F]{k-1}v_h$ is justified by the conservativity property $\diff[T]\GRAD u\SCAL\normal_{TF}+\diff[T']\GRAD u\SCAL\normal_{T'F}=0$ for all $F\in\Fhi$ with $\Th[F]=\{T,T'\}$, and the continuity property and boundary conditions expressed by \eqref{vem:cont.vh}--\eqref{vem:bc.vh}.
Setting $\term_2=\sum_{T\in\Th}\term_{2,T}$, we write
\begin{align}
	\term_{2,T}={}&-(\diff[T]\GRAD\dproj{k}u,\GRAD v_h)_T\nonumber\\
		={}&(\DIV(\diff[T]\GRAD\dproj{k}u),\lproj[T]{l}v_h)_T -\sum_{F\in\Fh[T]}(\diff[T]\GRAD\dproj{k}u\SCAL\normal_{TF},\lproj[F]{k-1}v_h)_F\nonumber\\
		={}&-(\diff[T]\GRAD u,\GRAD \lproj[T]{l}v_h)_T-\sum_{F\in\Fh[T]}(\diff[T]\GRAD\dproj{k}u\SCAL\normal_{TF},\lproj[F]{k-1}v_h-\lproj[T]{l}v_h)_F,
	\label{vem:T2}
\end{align}
where the first line follows from the definition \eqref{def:dproj.1} of the oblique elliptic projector with $v=v_h$ and $w=\dproj{k}u$, the second line is obtained by performing an integration-by-parts and using $\DIV(\diff[T]\GRAD\dproj{k}u)\in\Poly{l}(T)$ (since $l\ge k-2$) and $\diff[T]\GRAD\dproj{k}u\SCAL\normal_{TF}\in \Poly{k-1}(F)$ to replace $v_h$ by its projections on local element- and face-polynomial spaces, and the third line is a consequence of another integration-by-parts and of the definition \eqref{def:dproj.1} of the oblique elliptic projector which, applied to $v=u$ and $w=\lproj[T]{l}v_h$ (note that $l\le k$), gives $(\diff[T]\GRAD \dproj{k}u,\GRAD \lproj[T]{l}v_h)_T=(\diff[T]\GRAD u,\GRAD \lproj[T]{l}v_h)_T$. Hence, summing \eqref{vem:T2} over $T\in\Th$ and gathering with \eqref{vem:T1} yields
\[
\term_1+\term_2=\sum_{T\in\Th}\sum_{F\in\Fh[T]}(\diff[T](\GRAD \dproj{k}u-\GRAD u)\SCAL\normal_{TF},\lproj[T]{l}v_h-\lproj[F]{k-1}v_h)_F.
\]
We then estimate $\term_1+\term_2$ by using the Cauchy--Schwarz inequality:
\begin{align}
	|\term_1+\term_2|\le{}& \sum_{T\in\Th} \sum_{F\in\Fh[T]}\norm[F]{\diff[T]^{\frac12}(\GRAD\dproj{k}u-\GRAD u)}\overline{\lambda}_T^{\frac12}\norm[F]{\lproj[F]{k-1}(\lproj[T]{l}v_h-(v_h)_{|T})}
	\label{vem:T1T2.L2.dual}\\
	\lesssim{}& \sum_{T\in\Th} \overline{\lambda}_T^{\frac12}h_T^{r-\frac12}\seminorm[H^{r+1}(T)]{u}\overline{\lambda}_T^{\frac12}\norm[F]{\lproj[T]{l}v_h-(v_h)_{|T}}
\label{vem:T1T2.L2.primaldual}\\
	\lesssim{}& \sum_{T\in\Th} \overline{\lambda}_T^{\frac12}h_T^{r-\frac12}\seminorm[H^{r+1}(T)]{u}\overline{\lambda}_T^{\frac12}h_T^{\frac12}\norm[T]{\GRAD v_h},\nonumber
\end{align}
where we have used $l\le k-1$ together with the linearity and idempotency of $\lproj[F]{k-1}$ in the first line to write $\lproj[T]{l}v_h-\lproj[F]{k-1}v_h=\lproj[T]{l}v_h-\lproj[F]{k-1}(v_h)_{|T}=
\lproj[F]{k-1}(\lproj[T]{l}v_h-(v_h)_{|T})$, we passed to the second line by using
the trace approximation properties \eqref{dproj:trace} of $\dproj{k}$ (with $s=r+1$ and $m=0$)
and the $L^2(F)$-boundedness property of $\lproj[F]{k-1}$, and we concluded by invoking
the trace approximation property \eqref{eq:lproj:approx.trace} of $\lproj[T]{l}$ with $m=0$ and $s=1$.
Recalling the definition \eqref{def:alphaT} of $\alpha_T$, we have 
\begin{equation}\label{vem:est.gradvh}
	\overline{\lambda}_T^{\frac12}\norm[T]{\GRAD v_h}\le \overline{\lambda}_T^{\frac12}\underline{\lambda}_T^{-\frac12}\norm[T]{\diff[T]^{\frac12}\GRAD v_h}=\alpha_T^{\frac12}\norm[T]{\diff[T]^{\frac12}\GRAD v_h}\lesssim \alpha_T^{\frac12}a_T(v_h,v_h)^\frac12,
\end{equation}
where the last inequality is obtained introducing $\pm \diff[T]^{\frac12}\GRAD\dproj{k}v_h$ into the norm, using the triangle inequality, invoking the property \eqref{vem:def.sT} of $s_T$ with $w_h=v_h-\dproj{k}v_h$, and recalling the definition \eqref{vem:def.ah} of $a_T$.
Thus, using a Cauchy--Schwarz inequality on the sum over $T\in\Th$ and recalling the definition \eqref{vem:def.norm} of $\norm[\aXh]{{\cdot}}$, we conclude that
\begin{align}
	|\term_1+\term_2|\lesssim{}& \left(\sum_{T\in\Th} \alpha_T\overline{\lambda}_Th_T^{2r}\seminorm[H^{r+1}(T)]{u}^2\right)^\frac12 \norm[\aXh]{v_h}.
	\label{vem:T1T2}
\end{align}
\medskip\\
(ii) \emph{Term $\term_3$}. Apply \eqref{vem:IT.bound} with $\phi=u-\dproj{k}u$, which satisfies  $I_T \phi=I_T u-\dproj{k}u$ by the linearity of $I_T$ together with \textbf{(I2)}, to write
\begin{equation}\label{vem:IT.dproju}
	\norm[T]{\diff[T]^{\frac12}\GRAD(I_T u-\dproj{k}u)}\lesssim \overline{\lambda}_T^{\frac12}\left(\norm[T]{\GRAD(u-\dproj{k}u)}+h_T\seminorm[H^1(T)^d]{\GRAD (u-\dproj{k}u)}\right).
\end{equation}
A triangle inequality (introducing $\pm \diff[T]^{\frac12}\GRAD\dproj{k}u$ into the left-hand side) followed by \eqref{vem:IT.dproju}, the definition \eqref{def:alphaT} of $\alpha_T$, and the approximation properties \eqref{dproj:approx} of $\dproj[T]{k}$ with $s=r+1$ and $m=0$ and $m=1$ yield
\begin{align}
	\norm[T]{\diff[T]^{\frac12}\GRAD(I_T u-u)}\le{}&\norm[T]{\diff[T]^{\frac12}\GRAD(I_T u-\dproj[T]{k}u)}+	\norm[T]{\diff[T]^{\frac12}\GRAD(\dproj[T]{k}u-u)}\nonumber\\
\lesssim{}& \overline{\lambda}_T^{\frac12}
\left(\norm[T]{\GRAD(u-\dproj{k}u)}+h_T\seminorm[H^1(T)^d]{\GRAD (u-\dproj{k}u)}\right)\nonumber\\
	\lesssim{}&\alpha_T^{\frac12}\left(\norm[T]{\diff[T]^{\frac12}\GRAD(u-\dproj{k}u)}
+h_T\seminorm[H^1(T)^d]{\diff[T]^{\frac12}\GRAD(u-\dproj{k}u)}\right)\nonumber\\
		\lesssim{}& \alpha_T^{\frac12}\overline{\lambda}_T^{\frac12}h_T^r\seminorm[H^{r+1}(T)]{u}.
\label{vem:approx.ITu}
\end{align}
Applying the boundedness property \eqref{dproj:approx:3} of $\dproj[T]{k}$ to $\phi=I_T u-u$ and using \eqref{vem:approx.ITu} then leads to
\begin{equation}\label{vem:approx.dproj.ITu}
	\norm[T]{\diff[T]^{\frac12}\GRAD\dproj{k}(I_T u-u)}\lesssim \alpha_T^{\frac12}\overline{\lambda}_T^{\frac12}h_T^r\seminorm[H^{r+1}(T)]{u}.
\end{equation}
Using the Cauchy--Schwarz inequality, \eqref{vem:approx.dproj.ITu} along with \eqref{dproj:approx:3} for $\phi=v_h$, again a Cauchy--Schwarz inequality this time on the sum over $T\in\Th$, and \eqref{vem:est.gradvh} followed by the definition \eqref{vem:def.norm} of the norm $\norm[\aXh]{{\cdot}}$, we finally infer for the third term
\begin{equation}\label{vem:T3}
\begin{aligned}
	|\term_3|\le{}& \sum_{T\in\Th}\norm[T]{\diff[T]^\frac12\GRAD\dproj{k}(\aIh u-u)}\norm[T]{\diff[T]^\frac12\GRAD\dproj{k}v_h}\\
	\lesssim{}&\sum_{T\in\Th}\alpha_T^{\frac12}\overline{\lambda}_T^{\frac12}h_T^r\seminorm[H^{r+1}(T)]{u}
\norm[T]{\diff[T]^\frac12\GRAD v_h}
	\lesssim\left(\sum_{T\in\Th} \alpha_T\overline{\lambda}_Th_T^{2r}\seminorm[H^{r+1}(T)]{u}^2\right)^\frac12 \norm[\aXh]{v_h}.
\end{aligned}
\end{equation}
\medskip\\
(iii) \emph{Term $\term_4$}. We have
\begin{align}
	\Big|s_T\Big((I-\dproj{k})&\aIh u,(I-\dproj{k})v_h\Big)\Big|\nonumber\\
	\lesssim{}&s_T\Big((I-\dproj{k})\aIh u,(I-\dproj{k})\aIh u\Big)^{\frac12}s_T\Big((I-\dproj{k})v_h,(I-\dproj{k})v_h\Big)^{\frac12}\nonumber\\
	\lesssim{}&\norm[T]{\diff[T]^{\frac12}(\GRAD (I-\dproj{k})\aIh u)}\norm[T]{\diff[T]^{\frac12}\GRAD (v_h-\dproj[T]{k}v_h)}
	\label{vem:T4.L2}
\\
	\lesssim{}&\norm[T]{\diff[T]^{\frac12}(\GRAD (I-\dproj{k})\aIh u)}\norm[T]{\diff[T]^{\frac12}\GRAD v_h},
	\label{vem:T4.1}
\end{align}
where the first line follows from a Cauchy--Schwarz inequality, the second line is a consequence of \eqref{vem:def.sT}, and the third line is obtaind using the boundedness property \eqref{dproj:approx:3} of $\dproj{k}$. Introducing $\pm\diff[T]^{\frac12}\GRAD(u-\dproj{k}u)$ into the norm and using triangle inequalities, the first factor in the right-hand side of \eqref{vem:T4.1} is estimated by
\begin{align}
	\norm[T]{\diff[T]^{\frac12}\GRAD (I_T u-\dproj{k}I_T u)}\le{}&
	\norm[T]{\diff[T]^{\frac12}\GRAD (I_T u-u)}+
	\norm[T]{\diff[T]^{\frac12}\GRAD (u-\dproj{k}u)}+\norm[T]{\diff[T]^{\frac12}\GRAD \dproj{k}(u-I_T u)}\nonumber\\
	\lesssim{}& \alpha_T^{\frac12}\overline{\lambda}_T^{\frac12}h_T^r\seminorm[H^{r+1}(T)]{u},
\label{vem:dble.triangle}
\end{align}
where the conclusion follows from \eqref{vem:approx.ITu}, the approximation properties \eqref{dproj:approx} of $\dproj{k}$, and \eqref{vem:approx.dproj.ITu}. Plugging this estimate into \eqref{vem:T4.1}, summing over $T\in\Th$, using a Cauchy--Schwarz inequality on the sum over $T\in\Th$, and invoking \eqref{vem:est.gradvh} together with the definition \eqref{vem:def.norm} of $\norm[\aXh]{{\cdot}}$,
this yields
\begin{equation}\label{vem:T4}
	|\term_4|\lesssim \left(\sum_{T\in\Th}\alpha_T \overline{\lambda}_T h_T^{2r}\seminorm[H^{r+1}(T)]{u}^2\right)^{\frac12}\norm[\aXh]{v_h}.
\end{equation}
\medskip\\
\noindent (iv) \emph{Conclusion}. Plugging \eqref{vem:T1T2}, \eqref{vem:T3}, and \eqref{vem:T4} into \eqref{vem:cons.1} shows that $\norm[\aXh^\star]{\Cerr{u}{\cdot}}$ is bounded (up to a multiplicative constant) by the right-hand side of \eqref{vem:est.energy}, which concludes the proof of this inequality.

To establish \eqref{vem:est.recons}, we notice that, by definitions \eqref{vem:def.norm} of the norm on $V_{h,0}^k$ and \eqref{vem:def.ah} of $\abilh$,
\[
	\norm{\diff^{\frac12}(\GRADh\dproj[h]{k}u_h-\GRADh \dproj[h]{k}\aIh u)}\le \norm[\aXh]{u_h-\aIh u}\lesssim \left(\sum_{T\in\Th}\alpha_T \overline{\lambda}_T h_T^{2r}\seminorm[H^{r+1}(T)]{u}^2\right)^{\frac12}.
\]
The estimate \eqref{vem:est.recons} follows by introducing $\pm\diff^{\frac12}\GRADh(\dproj[h]{k}\aIh u-\dproj[h]{k}u)$ in its left-hand side, and by invoking \eqref{vem:approx.dproj.ITu} and the optimal approximation properties \eqref{dproj:approx} of the oblique elliptic projector, in a similar way as in \eqref{vem:dble.triangle}. \end{proof}

\begin{remark}[Unified analysis of conforming and non-conforming VEM]
  A unified analysis of conforming and non-conforming VEM based on an adaptation of the second Strang lemma has been recently proposed in \cite{Cangiani.Manzini.ea:17} in the context of more general second-order elliptic problems.
    
    A first difference with the present work is that, therein, the error is measured as $u-u_h$, the difference between the continuous and the virtual solutions.
    Thus, compared to Theorem \ref{a:th.est.var}, several additional terms have to be estimated in order to deduce an order of convergence from \cite[Theorem 2]{Cangiani.Manzini.ea:17}.
    These measure, in an appropriate way: the approximation properties of the virtual space $V_{h,0}^k$, those of the broken polynomial space $\Poly{k}(\Th)$, and the nonconformity of the method.
    
    A second difference with respect to the present work is that the dependence of the constants on the problem data is not specifically tracked.
    In the context of HHO methods, error estimates robust with respect to the problem data for second-order elliptic problems similar to the ones considered in \cite{Cangiani.Manzini.ea:17} have been recently proposed in \cite{Di-Pietro.Droniou.ea:15}; for a study of the links between HHO and non-conforming VEM we refer the reader to \cite{Cockburn.Di-Pietro.ea:16,Boffi.Di-Pietro:18,Di-Pietro.Droniou.ea:18}.
\end{remark}

\subsubsection{Improved error estimate in the $L^2$ norm}

\begin{theorem}[$L^2$ estimates for VEM schemes]\label{th:vem.L2}
Under the hypotheses of Theorem \ref{th:vem.energy}, assume moreover that $k\ge 2$,
that $l=1$ if $k=2$, that elliptic regularity holds  for \eqref{eq:strong}, and that
\begin{equation}\label{vem:proj.IT}
\lproj[T]{k-2}I_T \phi = \lproj[T]{k-2}\phi\quad\forall T\in\Th\,,\;\forall \phi\in H^1(T).
\end{equation}
Then, it holds that
\begin{equation}\label{vem:error.L2}
\norm{\dproj[h]{k}u_h-u}\lesssim h^{r+1}\seminorm[H^{r+1}(\Th)]{u},
\end{equation}
where the multiplicative constant additionally depends on $\diff$.
\end{theorem}

\begin{remark}[Assumption \eqref{vem:proj.IT}]
Assumption \eqref{vem:proj.IT} holds for both conforming and non-conforming VEM methods, as the moments
of degree $k-2$ in the cells are part of the DOFs of the methods, and $I_T \phi$ is defined as the element of $V_T^k$ that has the same DOFs as $\phi$.
\end{remark}

\begin{remark}[Dependency on the diffusion field in $L^2$ estimates]\label{rem:vem.L2}
  Elliptic regularity for problem \eqref{eq:strong} is only known if $\Omega$ is convex and $\diff$ is Lipschitz continuous.
  Combined with the assumption that $\diff$ is piecewise constant, this imposes $\diff$ constant over the entire domain, which means that we can treat anisotropic but not heterogeneous diffusion.
  For this reason, we make no attempt whatsoever to track the dependence on the diffusion field in the $L^2$ error estimate.
\end{remark}

\begin{proof} The elliptic regularity shows that, for all $g\in L^2(\Omega)$, $z_g\in H^2(\Omega)$ and
$\norm[H^2(\Omega)]{z_g}\lesssim \norm{g}$.
Estimate \eqref{vem:error.L2} therefore follows from \eqref{vem:est.energy}, Theorem \ref{th:abstract.weak} with the choice $r_h\coloneq \dproj[h]{k}$, and \eqref{vem:proj.IT} (which shows that $\dproj{k}I_T u-u=\dproj{k}I_T u-u-\lproj[T]{0}(\dproj{k}I_T u-u)$, whose $L^2(T)$-norm can be estimated using \eqref{eq:lproj:approx} and \eqref{vem:approx.dproj.ITu}), if we can prove that (with, as in the theorem, hidden constants in $\lesssim$ possibly depending on $\diff$) 
\begin{align}
\norm[\aXh^\star]{\Cerrdual{z_g}{\cdot}}\lesssim{}& h\norm[H^2(\Omega)]{z_g}
\label{vem:est.Cerrd}\\
|\Cerr{u}{\aIh z_g}|\lesssim{}& h^{r+1}\seminorm[H^{r+1}(\Th)]{u}\norm[H^2(\Omega)]{z_g}.
\label{vem:est.Cerrprimaldual}
\end{align}

\noindent (i) \emph{Dual consistency}. With our choice of $r_h$, we have
$\Cerrdual{z_g}{v_h}=(g,\dproj[h]{k}v_h)-\abilh(v_h,\aIh z_g)$. Since $\abilh$ is symmetric,
we see that $\Cerrdual{z_g}{v_h}$ is equal to $\Cerr{z_g}{v_h}$ in which the source term $(f,\lproj[h]{l}v_h)$ has been replaced with $(g,\dproj[h]{k}v_h)$. The estimate obtained in the proof of Theorem \ref{th:vem.energy} on the primal consistency error can therefore be used, with $r=1$ and $z_g\in H^2(\Omega)\subset H^{1+r}(\Th)$ instead of $u$, and yields \eqref{vem:est.Cerrd}, provided we examine the impact of changing $\lproj[h]{l}v_h$ into $\dproj[h]{k}v_h$.

The main difference between these two polynomials is that $\dproj[h]{k}v_h$ is a polynomial of degree $\le k$, whereas $\lproj[h]{l}v_h$ is a polynomial of degree $l\le k-1$. An inspection of the estimate of the primal consistency error shows that the only place where we used $\lproj[h]{l}v_h\in \Poly{k-1}(\Th)$ is in \eqref{vem:T1T2.L2.dual}, when estimating $\norm[F]{\lproj[T]{l}v_h-\lproj[F]{k-1}v_h}$.
Here, we therefore have to establish that, for $F\in\Fh[T]$,
\begin{equation}\label{vem:change.Cerrdual}
\norm[F]{\dproj[T]{k}v_h-\lproj[F]{k-1}v_h}\lesssim h_T^{\frac12}\norm[T]{\GRAD v_h}.
\end{equation}
We introduce $\pm\lproj[T]{k-1}v_h$ and write
\begin{align}
\norm[F]{\dproj[T]{k}v_h-\lproj[F]{k-1}v_h}
	\lesssim{}&
	\norm[F]{\dproj[T]{k}v_h-\lproj[T]{k-1}v_h}+\norm[F]{\lproj[T]{k-1}v_h-\lproj[F]{k-1}v_h}\nonumber\\
	\lesssim{}&
	h_T^{-\frac12}\norm[T]{\dproj[T]{k}v_h-\lproj[T]{k-1}v_h}+h_T^{\frac12}\norm[T]{\GRAD v_h},
\label{vem:change.Cerrdual.1}
\end{align}
where the first line is a triangle inequality, and the second line follows from a discrete trace inequality in $\Poly{k}(T)$ (which can be proved along the lines of \cite[Lemma 1.46]{Di-Pietro.Ern:12}) together with the arguments deployed after \eqref{vem:T1T2.L2.dual}. By \eqref{def:dproj.2} and since $\lproj[T]{k-1}v_h$ has the same average value over $T$ as $v_h$, we have $\dproj[T]{k}v_h-\lproj[T]{k-1}v_h=\dproj[T]{k}v_h-\lproj[T]{k-1}v_h-\lproj[T]{0}(\dproj[T]{k}v_h-\lproj[T]{k-1}v_h)$.
The approximation properties \eqref{eq:lproj:approx} of $\lproj[T]{0}$ with $s=1$, $m=0$ and $v=\dproj[T]{k}v_h-\lproj[T]{k-1}v_h$ then yield
\[
\norm[T]{\dproj[T]{k}v_h-\lproj[T]{k-1}v_h}\lesssim h_T\norm[T]{\GRAD(\dproj[T]{k}v_h-\lproj[T]{k-1}v_h)}\lesssim h_T \norm[T]{\GRAD v_h},
\]
where we have used the boundedness \eqref{dproj:approx:3} of $\dproj[T]{k}$, and the estimate $\norm[T]{\GRAD \lproj[T]{k-1}v_h}\lesssim \norm[T]{\GRAD v_h}$ (which follows from \eqref{eq:lproj:approx} with $s=m=1$).
Estimate \eqref{vem:change.Cerrdual} is then a consequence of \eqref{vem:change.Cerrdual.1}.

\medskip
\noindent (ii) \emph{Primal-dual consistency}. Following the discussion in Section \ref{sec:link.cons}, we re-visit the estimates on $\term_1,\ldots,\term_4$ in the proof of Theorem \ref{th:vem.energy}, and show that an additional $\mathcal O(h)$ factor can be obtained when $v_h=\aIh z_g$.

Let us start with the estimate \eqref{vem:T1T2.L2.primaldual} on $\term_1+\term_2$.
Recalling that $(v_h)_{|T}=I_T z_g$, introducing $\pm (\lproj[T]{l}z_g-z_g)$ into the norm, and using a triangle inequality, we can write
\begin{align*}
\norm[F]{\lproj[T]{l}v_h-(v_h)_{|T}}\le{}&
\norm[F]{\lproj[T]{l}(I_T z_g-z_g)-(I_T z_g-z_g)}+\norm[F]{\lproj[T]{l}z_g-z_g}\\
\lesssim{}&h_T^{\frac12}\seminorm[H^1(T)]{I_T z_g-z_g}+h_T^{\frac32}\seminorm[H^2(T)]{z_g},
\end{align*}
where the last line follows by applying \eqref{eq:lproj:approx.trace} with, for the first term,
$v=I_T z_g-z_g$, $s=1\le l+1$ and $m=0$ and, for the second term, $v=z_g$, $s=2\le l+1$ and $m=0$.
Invoking then \eqref{vem:approx.ITu} with $z_g$ instead of $u$ and $r=1$, we infer
$\norm[F]{\lproj[T]{l}v_h-v_h}\lesssim h_T^{\frac32}\seminorm[H^2(T)]{z_g}$. Plugged into 
\eqref{vem:T1T2.L2.primaldual}, this yields, thanks to a Cauchy--Schwarz inequality on the sum over $T\in\Th$,
\begin{equation}\label{vem:L2.T1T2}
|\term_1+\term_2|\lesssim \sum_{T\in\Th} h_T^{r+1}\seminorm[H^{r+1}(T)]{u}\seminorm[H^2(T)]{z_g}
\le h^{r+1}\seminorm[H^{r+1}(\Th)]{u}\seminorm[H^2(\Omega)]{z_g}.
\end{equation}

The term $\term_4$ is estimated starting from \eqref{vem:T4.L2}. Each term in the right-hand side of this estimate can be estimated by using \eqref{vem:dble.triangle} on $u$ for the first factor, and with $z_g$ instead of $u$ and $r=1$ for the second factor. Summing the resulting estimates over $T\in\Th$ and using a Cauchy--Schwarz inequality on the sum over $T\in\Th$ shows that
\begin{equation}\label{vem:L2.T4}
|\term_4|\lesssim \sum_{T\in\Th} h_T^r\seminorm[H^{r+1}(T)]{u}h_T\seminorm[H^2(T)]{z_g}
\le h^{r+1}\seminorm[H^{r+1}(\Th)]{u}\seminorm[H^2(\Omega)]{z_g}.
\end{equation}

We now turn to $\term_3$. Coming back to its definition in \eqref{vem:cons.1}, we have
$\term_3=\sum_{T\in\Th}\term_{3,T}$ with
\begin{align*}
\term_{3,T}={}&-(\diff[T]\GRAD(I_T u-u),\GRAD \dproj{k}I_Tz_g)_T\\
={}&(I_T u-u,\GRAD\SCAL(\diff[T]\GRAD \dproj{k}I_Tz_g))_T
-\sum_{F\in\Fh[T]}(I_T u-u,\diff[T]\GRAD \dproj{k}I_Tz_g\SCAL\normal_{TF})_F\\
={}&(\lproj[T]{k-2}(I_T u-u),\GRAD\SCAL(\diff[T]\GRAD \dproj{k}I_Tz_g))_T
-\sum_{F\in\Fh[T]}(\lproj[F]{k-1}(I_T u-u),\diff[T]\GRAD \dproj{k}I_Tz_g\SCAL\normal_{TF})_F,
\end{align*}
where we have used the definition \eqref{def:dproj.1} of $\dproj{k}(I_T u-u)$ with $w=\dproj[T]{k}I_T z_g$ in the first line, an integration by parts in the second line, and the fact that
$\GRAD\SCAL(\diff[T]\GRAD \dproj{k}I_Tz_g)\in\Poly{k-2}(T)$ and $\diff[T]\GRAD \dproj{k}I_Tz_g\SCAL\normal_{TF}\in\Poly{k-1}(F)$ to introduce the $L^2$-projections of $I_T u-u$ in the third line.
By Assumption \eqref{vem:proj.IT}, the first term in the right-hand side vanishes, and thus
$$
\term_{3}
=-\sum_{T\in\Th}\sum_{F\in\Fh[T]}(\lproj[F]{k-1}(I_T u-u),\diff[T]\GRAD (\dproj{k}I_Tz_g-z_g)\SCAL\normal_{TF})_F,
$$
where we have used the continuity property \eqref{vem:cont.vh} and the boundary condition \eqref{vem:bc.vh} on the functions in $V_{h,0}^k$, together with the continuity of the normal component of the flux $\diff\GRAD z_g$, to subtract 
$$\sum_{T\in\Th}\sum_{F\in\Fh[T]}(\lproj[F]{k-1}(I_T u-u),\diff[T]\GRAD z_g\SCAL\normal_{TF})_T=0.$$
A Cauchy--Schwarz inequality then gives
\begin{equation}\label{vem:L2.T3.1}
  |\term_{3}|
  \le\sum_{T\in\Th}\sum_{F\in\Fh[T]}
  \norm[F]{\lproj[F]{k-1}(I_T u-u)}
  \norm[F]{\diff[T]\GRAD (\dproj{k}I_Tz_g-z_g)\SCAL\normal_{TF}}.
\end{equation}
We next bound the factors inside the summation.
Assumption \eqref{vem:proj.IT} shows that $\lproj[T]{0}(I_T u-u)=0$ and thus, by \eqref{eq:lproj:approx.trace} with $s=1$ and $m=0$, we have for the first factor
\begin{align}
\norm[F]{\lproj[F]{k-1}(I_T u-u)}\le\norm[F]{I_T u-u}={}&
\norm[F]{(I_T u-u) - \lproj[T]{0}(I_T u-u)}\nonumber\\
\lesssim{}&h_T^\frac12 \norm[T]{\GRAD(I_T u-u)}\lesssim h_T^{\frac12+r}\seminorm[H^{r+1}(T)]{u},
\label{vem:L2.T3.3}
\end{align}
the conclusion following from \eqref{vem:approx.ITu}.
Introducing $\pm \diff[T]\GRAD\dproj[T]{k}z_g$ and using a triangle inequality, we get for the second factor
\begin{align}
\norm[F]{\diff[T]\GRAD (\dproj{k}I_Tz_g-z_g)\SCAL\normal_{TF}}\lesssim{}&
\norm[F]{\GRAD \dproj{k}(I_Tz_g-z_g)}+\norm[F]{\GRAD(\dproj{k}z_g-z_g)}\nonumber\\
\lesssim{}&h_T^{-\frac12}\norm[T]{\GRAD \dproj{k}(I_Tz_g-z_g)}+h_T^{\frac12}\seminorm[H^2(T)]{z_g}\lesssim h_T^{\frac12}\seminorm[H^2(T)]{z_g},
\label{vem:L2.T3.2}
\end{align}
where we have used a discrete trace inequality in $\Poly{k-1}(T)$ and \eqref{dproj:trace} (with $v=z_g$, $s=2\le k+1$
and $m=0$) to pass to the second line, and we have concluded by invoking \eqref{vem:approx.dproj.ITu} with $z_g$ instead of $u$ and $r=1$.
Plugging \eqref{vem:L2.T3.2} and \eqref{vem:L2.T3.3} into \eqref{vem:L2.T3.1}, we obtain
\[
|\term_3|\lesssim \sum_{T\in\Th} h_T^{r+\frac12}\seminorm[H^{r+1}(T)]{u}h_T^\frac12\seminorm[H^2(T)]{z_g}
\lesssim h^{r+1}\seminorm[H^{r+1}(\Th)]{u}\seminorm[H^2(\Omega)]{z_g}.
\]
Together with \eqref{vem:L2.T1T2} and \eqref{vem:L2.T4}, this establishes \eqref{vem:est.Cerrprimaldual} and concludes
the proof. \end{proof}

\begin{remark}[Simplifications]
The proofs of Theorems \ref{th:vem.energy} and \ref{th:vem.L2} have been made in a unified setting that covers
both conforming and non-conforming VEM.
Simplifications are possible when those methods are considered individually.

For non-conforming VEM \cite{Ayuso-de-Dios.Lipnikov.ea:16}, significant simplifications stem from the following preservation properties of the interpolant: for all $T\in\Th$ and $\phi\in H^1(T)$,
\[
\lproj[T]{k-2}I_T\phi=\lproj[T]{k-2}\phi\quad\mbox{ and }\quad \lproj[F]{k-1}I_T\phi=\lproj[F]{k-1}\phi\quad\forall F\in\Fh[T].
\]
Performing integrations-by-parts on the definition \eqref{def:dproj.1} of $\dproj[T]{k}$, it can easily be seen that these properties imply $\dproj[T]{k}I_T\phi=\dproj[T]{k}\phi$.
As a consequence, the term $\term_3$ entirely vanishes, and a few other estimates are
shorter (e.g., \eqref{vem:dble.triangle} is a direct consequence of \eqref{vem:IT.dproju} and \eqref{dproj:approx}, etc.).
Note that $\term_3$ is by far the most troublesome term to estimate in the proof of Theorem \ref{th:vem.L2}.

In the context of conforming VEM, on the other hand, a slightly simpler argument can be invoked working in a more standard setting corresponding to the classical first Strang lemma; see, e.g., \cite[Lemma 3.11]{Brenner.Guan.ea:17}.
  In this case, the source term for the dual problem is the error $u-u_h$ measured as the difference between the continuous and virtual solutions.

  We close this remark by noticing that, unlike \cite[Theorem 6]{Cangiani.Manzini.ea:17} and \cite[Theorem 3.14]{Brenner.Guan.ea:17}, our $L^2$-error estimate stems from an application of the abstract result of Theorem \ref{th:abstract.weak}, which is not problem-specific.
\end{remark}

\begin{remark}[The lowest-order case]
  As is often the case with mixed and non-conforming methods for diffusion equations on generic grids, $L^2$ error estimates for the lowest degree(s) require specific work and, possibly, additional regularity on the source term; see, e.g., ~\cite{Di-Pietro.Ern.ea:14,Di-Pietro.Ern:17} for primal and mixed HHO methods, \cite{Beirao-da-Veiga.Brezzi.ea:13,Ayuso-de-Dios.Lipnikov.ea:16} for conforming and non-conforming VEM, \cite[Remark 8.5]{Boffi.Di-Pietro:18} for further insight into this topic, and \cite[Section 2.7]{Lipnikov.Manzini:14} for a fix in the context of high-order Mimetic Finite Difference methods.
  In Theorem \ref{th:abstract.weak}, additional work would be required on the primal-dual consistency error. The details can be evinced from the above references, and are omitted here for the sake of brevity.
\end{remark}

\subsection{Finite Volume methods}\label{sec:fv}

The second application of the abstract analysis framework of Section \ref{sec:abstract.analysis} considered here is to Finite Volume (FV) methods.
  In this context, several novelties are present.
  First of all, the analysis is carried out under general assumptions on the numerical fluxes, which enables the simultaneous treatment of several (cell-centred or cell- and face-centred) schemes.
  Second, we provide a clear definition of consistency also for FV schemes for which this notion hadn't been clearly highlighted in the literature.
  Third, to the best of our knowledge, we write the first error estimates, for FV methods, in which the dependency on the diffusion field is finely tracked.

\subsubsection{General theory}\label{sec:fv.general}

The discrete unknowns of Finite Volume methods are usually values at points. We consider here methods with cell- and face-unknowns (see Section \ref{sec:fv.mpfa} for cell-centred methods). A mesh $\Mh=(\Th,\Fh)$ being chosen, we therefore take one point $\vec{x}_T$ in each cell $T\in\Th$ and one point $\vec{x}_F$ on each face $F\in\Fh$; note that these points may not be the centres of mass of the corresponding geometrical objects, and may need to satisfy
specific geometric properties. The $(d-1)$-dimensional measure of a face $F\in\Fh$ is denoted by $|F|$ and, if $T\in\Th[F]$, $d_{TF}$ is the orthogonal distance between $\vec{x}_T$ and $F$.

The space of unknowns is
\[
\aXh\coloneq\left\{v_h=((v_T)_{T\in\Th},(v_F)_{F\in\Fh})\st
v_T\in\Real\;\forall T\in\Th,
\quad v_F\in\Real\;\forall F\in\Fhi,
\quad v_F=0\;\forall F\in\Fhb\right\},
\]
which is equipped with the following discrete equivalent of the $H^1_0$-norm:
\begin{equation}\label{fv:norm.aXh}
  \norm[1,\Th]{v_h}\coloneq\left(\sum_{T\in\Th}\underline{\lambda}_T\seminorm[1,T]{v_h}^2\right)^{\frac12} \quad\mbox{ with }
\quad \seminorm[1,T]{v_h}^2\coloneq\sum_{F\in\Fh[T]}|F|d_{TF}\left(\frac{v_T-v_F}{d_{TF}}\right)^2,
\end{equation}
 (see, e.g., \cite[Section 7.1]{Droniou.Eymard.ea:18} -- note that, contrary to this reference, we explicitly account for the diffusion coefficient here).
For $u\in C(\overline{\Omega})$ with $u_{|\partial\Omega}=0$, an interpolant $\aIh u\in\aXh$ is defined by setting
\[
\aIh u=\Big((u(\vec{x}_T))_{T\in\Th},(u(\vec{x}_F))_{F\in\Fh}\Big).
\]
Note that, in dimensions $\le 3$, the solution $u$ to \eqref{eq:weak} is (H\"older) continuous on
$\overline{\Omega}$ \cite{Stampacchia:65}.

FV methods are characterised by flux conservativity and balance equations. Following the pre\-sen\-tation in \cite{Droniou:14}, a generic FV method for \eqref{eq:strong} is written:
Find $u_h\in \aXh$ such that
\begin{subequations}
\label{scheme:fv}
	\begin{align}
	\Fl(u_h) + \Fl[T'F](u_h)=0&\qquad\forall F\in\Fhi\mbox{ with }\Th[F]=\{T,T'\},\label{fv:conservativity}\\
	\sum_{F\in\Fh[T]}\Fl(u_h) = \int_T f&\qquad\forall T\in\Th. \label{fv:balance}
	\end{align}
\end{subequations}
Here, for $T\in\Th$ and $F\in\Fh[T]$, $\Fl:\aXh\to\Real$ is a linear numerical flux such that $\Fl(\aIh u)$ approximates $-\int_F \diff\GRAD u\SCAL\normal_{TF}$.

The following general estimate is a direct consequence of Theorem \ref{a:th.est.var}.

\begin{theorem}[Energy estimate for FV methods]\label{th:fv.ener.est}
Assume that the fluxes $(\Fl)_{T\in\Th,\,F\in\Fh[T]}$ satisfy the following coercivity property, for some $\gamma>0$:
For all $v_h\in\aXh$,
\begin{equation}\label{fv:coer}
	\sum_{T\in\Th}\sum_{F\in\Fh[T]}\Fl(v_h)(v_T-v_F)\ge \gamma\norm[1,\Th]{v_h}^2.
\end{equation}
Then, if the solution $u$ to \eqref{eq:weak} belongs to $C(\overline{\Omega})\cap H^2(\Th)$, denoting by $u_h$ the solution to the FV scheme \eqref{scheme:fv}, it holds
\begin{equation}\label{eq:fv.ener.est}
\norm[1,\Th]{u_h-\aIh u}\le \gamma^{-1}\left(\sum_{T\in\Th}\underline{\lambda}_T^{-1}\sum_{F\in\Fh[T]}\frac{d_{TF}}{|F|}\left[\int_F \diff[T]\GRAD u_{|T}\SCAL\normal_{TF}+\Fl(\aIh u)\right]^2\right)^{\frac12}.
\end{equation}
\end{theorem}

\begin{remark}[Consistency of the fluxes]
Estimate \eqref{eq:fv.ener.est} highlights the following well-known fact (see
\cite[Example 3.1]{Eymard.Gallouet.ea:00} or \cite[Remark 1.3]{Droniou:14}):
in FV methods, the appropriate consistency is that of the \emph{fluxes}, not of the discrete second order differential operator as in Finite Difference methods .
\end{remark}

\begin{proof}
We first recast problem \eqref{scheme:fv} under a discrete weak form.
For an arbitrary vector $v_h=((v_T)_{T\in\Th},(v_F)_{F\in\Fh})\in \aXh$, notice that, by the
flux conservativity \eqref{fv:conservativity} and the boundary condition on $v_h$,
\begin{equation}\label{fv:ener.0}
\begin{aligned}
  \sum_{T\in\Th}\sum_{F\in\Fh[T]}\Fl(u_h)v_F={}&
  \sum_{F\in\Fhi,\,\Th[F]=\{T,T'\}}\left(\Fl(u_h)+\Fl[T'F](u_h)\right)v_F\\
&+\sum_{F\in\Fhb,\,\Th[F]=\{T\}}\Fl(u_h)v_F=0,
\end{aligned}
\end{equation}
where the first equality comes from a re-arrangement of the sum over the faces.
Hence, multiplying \eqref{fv:balance} by $v_T$, summing over $T\in\Th$ and using the above relation, we see that $u_h$ satisfies
\begin{equation}\label{fv:weak}
	\underbrace{\sum_{T\in\Th}\sum_{F\in\Fh[T]}\Fl(u_h)(v_T-v_F)}_{\abilh(u_h,v_h)}=\underbrace{\sum_{T\in\Th}\int_T f v_T}_{\alinh(v_h)}\qquad\forall v_h\in\aXh.
\end{equation}
This problem has the form \eqref{a:pro.h} with $\aXh=\aYh$. The coercivity assumption \eqref{fv:coer} shows that $\abilh$ is coercive on $\aXh$, with coercivity constant $\gamma$. Hence, Theorem
\ref{a:th.est.var} yields
\begin{equation}\label{fv:ener.1}
\norm[1,\Th]{u_h-\aIh u}\le \gamma^{-1} \norm[\aXh^\star]{\Cerr{\aIh u}{\cdot}}.
\end{equation}
To estimate the primal consistency error, notice first that the relation $f=-\DIV(\diff\GRAD u)$ and the divergence formula in each cell give
\begin{align*}
\alinh(v_h)=\sum_{T\in\Th} \left(\int_T -\DIV(\diff\GRAD u)\right) v_T
={}&-\sum_{T\in\Th} \sum_{F\in\Fh[T]} \left(\int_F \diff[T]\GRAD u_{|T}\SCAL\normal_{TF}\right)v_T\\
={}&-\sum_{T\in\Th} \sum_{F\in\Fh[T]} \left(\int_F \diff[T]\GRAD u_{|T}\SCAL\normal_{TF}\right)(v_T-v_F)
\end{align*}
where we have used \eqref{fv:ener.0} with $\Fl(u_h)$ replaced with $\int_F \diff[T]\GRAD u_{|T}\SCAL\normal_{TF}$ (these exact fluxes also satisfy the conservativity relation \eqref{fv:conservativity} since $\DIV(\diff\GRAD u)\in L^2(\Omega)$). Hence,
\[
	\Cerr{\aIh u}{v_h}=-\sum_{T\in\Th}\sum_{F\in\Fh[T]}\left[\int_F \diff[T]\GRAD u_{|T}\SCAL\normal_{TF}+\Fl(\aIh u)\right](v_T-v_F).
\]
A Cauchy--Schwarz inequality and the definition \eqref{fv:norm.aXh} of the norm on $\aXh$ shows that $\norm[\aXh^\star]{\Cerr{\aIh u}{\cdot}}$ is bounded above by the bracketed term in the right-hand side of \eqref{eq:fv.ener.est}. Plugging into \eqref{fv:ener.1} this bound of the primal consistency error concludes the proof.
 \end{proof}

\subsubsection{Stable and linearly exact fluxes}

The estimate \eqref{eq:fv.ener.est} enables us to identify simple local properties on the fluxes,
under which an $\mathcal O(h)$ energy estimate can be established: local dependency, linear
exactness and boudedness.
Similar properties were proposed in \cite{Droniou.Eymard:17}, but without the concept of local dependency, which is essential for establishing a proper error estimate. Additionally, the analysis in \cite{Droniou.Eymard:17} was only sketched, and did not track the dependency of the estimates on the diffusion tensor $\diff$.

In this section, for $T\in\Th$ we let
$X_T\coloneq\left\{v=(v_T,(v_F)_{F\in\Fh[T]})\,:\,v_T\in\Real\,,\;v_F\in\Real\quad\forall F\in \Fh[T]\right\}$ be the local space of unknowns and, for $\phi\in C(\overline{T})$, $I_T\phi=(\phi(\vec{x}_T),(\phi(\vec{x}_F))_{F\in\Fh[T]})\in X_T$
defines the local interpolant of $\phi$.

\begin{theorem}[Energy error estimate for linearly exact FV methods]\label{th:fv.energy.linex}
Assume that the family of numerical fluxes $(\Fl)_{T\in\Th,\,F\in\Fh[T]}$ satisfies the coercivity property \eqref{fv:coer}, as well as the following
properties:
\begin{enumerate}[(i)]
\item \emph{Local dependency and linear exactness.} For all $v_h\in \aXh$, $T\in\Th$ and $F\in\Fh[T]$, $\Fl(v_h)$ depends only on the values $(v_{T},(v_{F})_{F\in\Fh[T]})\in X_T$, and if $L$ is an affine function on $\overline{T}$ then $\Fl(I_T L)=-\int_F \diff[T]\GRAD L\SCAL\normal_{TF}$.
\item \emph{Boundedness.} There is $C_b\ge 0$ such that, for all $T\in\Th$ and $v\in X_T$,
\begin{equation}\label{fv:bound}
	\sum_{F\in\Fh[T]}\frac{d_{TF}}{|F|}|\Fl(v)|^2\le C_b \overline{\lambda}_{T}^2\seminorm[1,T]{v}^2.
\end{equation}
\end{enumerate}
Let
\[
	\theta\ge \max_{T\in\Th}\left(\max_{F\in\Fh[T]}\frac{h_T}{d_{TF}}+{\rm Card}(\Fh[T])\right).
\]
Then, if the solution $u$ to \eqref{eq:weak} belongs to $C(\overline{\Omega})\cap H^2(\Th)$, denoting by $u_h$ the solution of \eqref{scheme:fv}, 
\begin{equation}\label{fv:est.energy}
\norm[1,\Th]{u_h-\aIh u}\lesssim \gamma^{-1}\left(\sum_{T\in\Th} \alpha_T\overline{\lambda}_T h_T^2
\seminorm[H^2(T)]{u}^2\right)^{\frac12},
\end{equation}
with hidden constant independent on $\diff$ and $h$, but depending on $\theta$ and $C_b$.
\end{theorem}

\begin{proof}
Fix $T\in\Th$ and notice that, by definition of $\theta$ and \cite[Lemma B.1]{Droniou.Eymard.ea:18}, there is a ball of radius $\gtrsim h_T$ such that $T$ is star-shaped with respect to all points in this ball. Hence, \cite[Lemma 7.61]{Droniou.Eymard.ea:18} yields the existence of a linear function $L_T$ such that, setting $R_T=u_{|T}-L_T$,
\begin{equation}\label{est:RT}
\sup_{\overline{T}}|R_T|\lesssim h_T^{2-\frac{d}{2}}\seminorm[H^2(T)]{u}\quad
\mbox{ and }\quad \norm[T]{\GRAD R_T}\lesssim h_T\seminorm[H^2(T)]{u}.
\end{equation}
Subtracting $L_T$ and using the linear exactness of the fluxes, we have
\begin{align*}
\term_{TF}\coloneq{}\left|\int_F \diff[T]\GRAD u_{|T}\SCAL\normal_{TF}+\Fl(I_T u_{|T})\right|
={}& \left|\int_F \diff[T]\GRAD R_T\SCAL\normal_{TF}+\Fl(I_T R_T)\right|\\
\le{}& \overline{\lambda}_T \int_F |\GRAD R_T| + |\Fl(I_TR_T)|.
\end{align*}
Hence, by boundedness of the fluxes,
\[
\sum_{F\in\Fh[T]}\frac{d_{TF}}{|F|}\term_{TF}^2\le 2
\overline{\lambda}_T^2\sum_{F\in\Fh[T]}d_{TF}|F| \left(\frac{1}{|F|}\int_F |\GRAD R_T|\right)^2
+2C_b\overline{\lambda}_T^2\seminorm[1,T]{I_T R_T}^2=:
\term_T^{(1)}+\term_T^{(2)}.
\]
The definition of $\term_{TF}$ and Theorem \ref{th:fv.ener.est} show that
\begin{equation}\label{fv:linex.1}
\norm[1,\Th]{u_h-\aIh u}\le \gamma^{-1}\left(\sum_{T\in\Th}\underline{\lambda}_T^{-1}(\term_T^{(1)}+\term_T^{(2)})\right)^{\frac12}.
\end{equation}
To estimate $\term_T^{(1)}$, we apply \cite[Lemma B.6]{Droniou.Eymard.ea:18} to $|\GRAD R_T|\in H^1(T)$ to see that
\[
\left(\frac{1}{|F|}\int_F |\GRAD R_T|\right)^2\lesssim 
\left(\frac{1}{|T|}\int_T |\GRAD R_T|\right)^2 + \frac{h_T}{|F|}\seminorm[H^2(T)]{R_T}^2.
\]
Since $L_T$ is linear, $\seminorm[H^2(T)]{R_T}=\seminorm[H^2(T)]{u-L_T}=\seminorm[H^2(T)]{u}$. Hence,
the Jensen inequality on the first term in the right-hand side and \eqref{est:RT} yield
\[
\left(\frac{1}{|F|}\int_F |\GRAD R_T|\right)^2\lesssim 
\left(\frac{h_T^2}{|T|}+\frac{h_T}{|F|}\right)\seminorm[H^2(T)]{u}^2.
\]
Plugging this bound into the definition of $\term_T^{(1)}$, using $d_{TF}\le h_T$, and using $\sum_{F\in\Fh[T]}d_{TF}|F|=d|T|$ (see \cite[Lemma B.2]{Droniou.Eymard.ea:18}), we infer
\begin{equation}\label{fv:linex.2}
\term_T^{(1)}\lesssim \overline{\lambda}_T^2 h_T^2 \seminorm[H^2(T)]{u}^2.
\end{equation}
For $\term_T^{(2)}$, we recall the definition of $\seminorm[1,T]{{\cdot}}$, use the first bound in \eqref{est:RT}, and the estimates $\frac{1}{d_{TF}}\le \frac{\theta}{h_T}$ and $|F|\lesssim h_T^{d-1}$
to write
\[
\term_T^{(2)} \lesssim \overline{\lambda}_T^2 \sum_{F\in\Fh[T]}\frac{|F|}{d_{TF}} (|R_T(\vec{x}_T)|^2+|R_T(\vec{x}_F)|^2)
\lesssim \overline{\lambda}_T^2 h_T^{2}\seminorm[H^2(T)]{u}^2.
\]
Using this estimate together with \eqref{fv:linex.2} into \eqref{fv:linex.1} concludes the proof. 
\end{proof}

We now give two classical examples of FV methods that satisfy the coercivity, linear exactness and stability properties, and to which Theorem \ref{th:fv.energy.linex} thus applies.
Error estimates for these two methods can be found in the literature (see e.g.\ \cite{Eymard.Gallouet.ea:00,Droniou.Eymard.ea:10}) but, to our best knowledge, contrary to \eqref{fv:est.energy}, none of the currently available estimate has explicit dependency on the local anisotropy ratio and diffusion magnitude.

\begin{example}[Two-Point Flux Approximation (TPFA) method]
The TPFA scheme \cite{Eymard.Gallouet.ea:00} requires meshes with a specific geometric property:
the points $(\vec{x}_T)_{T\in\Th}$ and $(\vec{x}_F)_{F\in\Fh}$ must be chosen such that, for
any $T\in\Th$ and $F\in\Fh[T]$, $\vec{x}_T\vec{x}_F$ is parallel to $\diff[T]\normal_{TF}$.
The fluxes are then defined by: for $v_h\in\aXh$,
\begin{equation}\label{def:TPFA}
\Fl(v_h)=|F|\,|\diff[T]\normal_{TF}|\frac{v_T-v_F}{|\vec{x}_T-\vec{x}_F|}.
\end{equation}
The assumption on the points show that $\vec{x}_F-\vec{x}_T=\alpha_{TF}\diff[T]\normal_{TF}$ with $\alpha_{TF}>0$ (because $\diff[T]$ is symmetric positive definite and $(\vec{x}_F-\vec{x}_T)\SCAL
\normal_{TF}>0$). Taking the norm on both sides yields $\alpha_{TF}=\frac{|\vec{x}_T-\vec{x}_F|}{|\diff[T]\normal_{TF}|}$. Hence, if $L$ is a linear function, 
\[
L(\vec{x}_T)-L(\vec{x}_F)=\GRAD L\SCAL (\vec{x}_T-\vec{x}_F)
=-\frac{|\vec{x}_T-\vec{x}_F|}{|\diff[T]\normal_{TF}|} \GRAD L\SCAL \diff[T]\normal_{TF}
\]
and thus $\Fl(I_T L)=-|F|\diff[T]\GRAD L\SCAL\normal_{TF}$, showing that the flux
is linearly exact. Fixing $C_b\ge \max_{T\in\Th}\max_{F\in\Fh[T]}\frac{d_{TF}^2}{|\vec{x}_T-\vec{x}_F|^2}$,
the boundedness property \eqref{fv:bound} is a straightforward consequence of \eqref{def:TPFA}.
Since $|\diff[T]\normal_{TF}|\ge \underline{\lambda}_T$, the coercivity \eqref{fv:coer} also easily 
follows from \eqref{def:TPFA}, provided that $\gamma>0$ is chosen such that
$\gamma\le \min_{T\in\Th}\min_{F\in\Fh[T]}\frac{d_{TF}}{|\vec{x}_T-\vec{x}_F|}$.
\end{example}

\begin{example}[Mixed Finite Volume (MFV) method]
The MFV method is the FV presentation of the Hybrid Mimetic Mixed (HMM) method \cite{Droniou.Eymard.ea:10}.
Here, $(\vec{x}_T)_{T\in\Th}$ can be any points in the cells, but $(\vec{x}_F)_{F\in\Fh}=(\overline{\vec{x}}_F)_{F\in\Fh}$ are taken as the centers of mass of the faces. To construct the MFV method \cite{Droniou.Eymard:06,Droniou.Eymard.ea:10}, we start by reconstructing, from known fluxes, a local gradient. For $T\in\Th$, if $\flgen_T=(\flgen_{TF})_{F\in\Fh[T]}$ is a family of real numbers (representing fluxes through the faces of $T$), define the following discrete gradient and boundary residuals:
\begin{align*}
\vec{G}_T(\flgen_T)\coloneq{}&-\frac{1}{|T|}\diff[T]^{-1}\sum_{F\in\Fh[T]}\flgen_{TF} (\overline{\vec{x}}_F-\vec{x}_T),\\
\mathcal R_{TF}(\flgen_T)\coloneq{}& \flgen_{TF}+|F|\diff[T]\vec{G}_T(\flgen_T)\SCAL\normal_{TF}\quad\forall F\in\Fh[T].
\end{align*}
Then, fixing a symmetric positive definite matrix $\mathbb{B}^T=(\mathbb{B}^T_{FF'})_{F,F'\in\Fh[T]}\in
\Real^{\Fh[T]\times \Fh[T]}$, the MFV fluxes $(\Fl(v_h))_{F\in\Fh[T]}$ are defined, for $v_h\in\aXh$, as the unique solution of the following problem
\begin{align}
&\forall \flgen_T=(\flgen_{TF})_{F\in\Fh[T]}\in\Real^{\Fh[T]},\nonumber\\
&|T|\diff[T]\vec{G}_T(\Fl(v_h))\SCAL\vec{G}_T(\flgen_T)+\sum_{F,F'\in\Fh[T]}\mathbb{B}^T_{FF'}
\mathcal R_{TF}(\Fl(v_h))\mathcal R_{TF}(\flgen_T)=\sum_{F\in\Fh[T]}(v_T-v_F)\flgen_{TF}.
\label{mfv:def.fl}
\end{align}
Assume that $L$ is a linear map and that $v_h=\aIh L$. Let $\mathfrak{g}_T=(-|F|\diff[T]\GRAD L\SCAL\normal_{TF})_{F\in\Fh[T]}$ be the exact fluxes of $L$. The divergence theorem shows that $\vec{G}_T(\mathfrak{g}_T)=\GRAD L$ and thus $\mathcal R_{TF}(\mathfrak{g}_T)=0$. Moreover, for all $\flgen_T\in\Real^{\Fh[T]}$, 
\[
\sum_{F\in\Fh[T]}(v_T-v_F)\flgen_{TF}=\sum_{F\in\Fh[T]}\GRAD L\SCAL (\vec{x}_T-\overline{\vec{x}}_F)\flgen_{TF}=\GRAD L \SCAL |T|\diff[T]\vec{G}_T(\flgen_T).
\]
Hence, \eqref{mfv:def.fl} holds with $\mathfrak{g}_{T}$ instead of $(\Fl(v_h))_{F\in\Fh[T]}$, which shows that these two families of fluxes are equal, and thus that the fluxes are linearly exact. The stability and coercivity of the method follow easily from \eqref{mfv:def.fl}, under natural assumption on the matrices $\mathbb{B}^T$, see \cite[Section 4.1]{Droniou.Eymard.ea:10} or \cite[Chapter 13]{Droniou.Eymard.ea:18}.
\end{example}

\begin{remark}[$L^2$ estimates and super-convergence]\label{rem:L2.FV}
Define $r_h:\aXh\to L^2(\Omega)$ by $(r_h v_h)_{|T}=v_T$ for all $v_h\in\aXh$ and $T\in\Th$.
A discrete Poincar\'e inequality \cite[Remark B.16]{Droniou.Eymard.ea:18} yields $\norm{r_h v_h}\le C \norm[1,\Th]{v_h}$, with $C$ depending only on $\eta\ge \max_{F\in\Fhi,\,\Th[F]=\{T,T'\}}\left(\frac{d_{TF}}{d_{T'F}}+\frac{d_{T'F}}{d_{TF}}\right)$.
Hence, \eqref{eq:fv.ener.est} and \eqref{fv:est.energy} directly give estimates on $\norm{r_h u_h - u_{\Th}}$, where $u_{\Th}=r_h\aIh u$ is the piecewise constant function defined by $(u_{\Th})_{|T}=u(\vec{x}_T)$ for all $T\in\Th$.

One can naturally wonder whether Theorem \ref{th:abstract.weak} could yield better error estimates on
this $L^2$-norm. The answer is no in general. Numerical test 2 in \cite{Droniou.Nataraj:18} shows that, for the MFV scheme, the $L^2$-norm error can, in some cases, converge at the same rate as the discrete energy error (that is, $\mathcal O(h)$). Actually, for the MFV and TPFA schemes at least, the super-convergence properties in $L^2$-norm seem to be related to the proximity, locally and on average, of the interpolation points $(\vec{x}_T)_{T\in\Th}$ and the centers of mass of the cells \cite[Theorem 5.3]{Droniou.Nataraj:18}.
\end{remark}

\subsubsection{Cell-centred methods, application to Multi-Point Flux Approximations}\label{sec:fv.mpfa}

The theory in Section \ref{sec:fv.general} can easily be adapted to purely cell-centred methods.
For such methods, the space of unknowns is
\[
\aXh^c\coloneq \Poly{0}(\Th)=\left\{v_h=(v_T)_{T\in\Th}\st v_T\in\Real\right\},
\]
with discrete $H^1_0$ norm defined by
\[
\norm[1,\Th,c]{v_h}\coloneq \left(\sum_{F\in\Fh}\underline{\lambda}_F |F|d_F\left(\frac{v_T-v_{T'}}{d_F}\right)^2\right)^{\frac12},
\]
with the notations
\[
\begin{array}{ll}
\forall F\in\Fhi\st \underline{\lambda}_F=\min(\underline{\lambda}_T,\underline{\lambda}_{T'})\mbox{ and } d_F=d_{TF}+d_{T'F}\,,\;\mbox{ where }\{T,T'\}=\Th[F],\\
\forall F\in\Fhb\st \underline{\lambda}_F=\underline{\lambda}_T\,,\;d_F=d_{TF}\mbox{ and }v_{T'}=0,\;\mbox{ where }\{T\}=\Th[F].
\end{array}
\]
The interpolant of a continuous function $u$ is $\aIh^c u\coloneq(u(\vec{x}_T))_{T\in\Th}\in\aXh^c$. To write a cell-centred FV method, linear fluxes $\Fl^c:\aXh^c\to\Real$ are first chosen such that
\begin{equation}\label{fv:cc:cons}
\Fl^c+\Fl[T'F]^c=0\mbox{ on $\aXh^c$},\quad \forall F\in\Fhi\mbox{ with $\Th[F]=\{T,T'\}$}.
\end{equation}
Then the FV scheme reads: Find $u_h\in\aXh^c$ such that
\begin{equation}\label{fv:balance.c}
\sum_{F\in\Fh[T]}\Fl^c(u_h) = \int_T f\qquad\forall T\in\Th.
\end{equation}

The following result is the equivalent for cell-centred methods of Theorem \ref{th:fv.ener.est}.

\begin{theorem}[Energy estimate for cell-centred FV methods]\label{th:fv.ener.est.cc}
Assume that the fluxes $(\Fl^c)_{T\in\Th,\,F\in\Fh[T]}$ satisfy the following coercivity property, for some $\gamma>0$:
For all $v_h\in\aXh^c$,
\begin{equation}\label{fv:coer.cc}
	\sum_{F\in\Fhi,\,\Th[F]=\{T,T'\}}\Fl^c(v_h)(v_T-v_{T'})+
	\sum_{F\in\Fhb,\,\Th[F]=\{T\}}\Fl^c(v_h)v_T\ge \gamma\norm[1,\Th,c]{v_h}^2.
\end{equation}
Then, if the solution $u$ to \eqref{eq:weak} belongs to $C(\overline{\Omega})\cap H^2(\Th)$, denoting by $u_h$ the solution to the FV scheme \eqref{fv:balance.c}, it holds
\begin{equation}\label{eq:fv.ener.est.cc}
\norm[1,\Th,c]{u_h-\aIh^c u}\le \gamma^{-1}\left(\sum_{F\in\Fh}\underline{\lambda}_F^{-1}\frac{d_{F}}{|F|}\left[\int_F \diff[T]\GRAD u_{|T}\SCAL\normal_{TF}+\Fl^c(\aIh^c u)\right]^2\right)^{\frac12}
\end{equation}
where, for $F\in\Fh$, $T$ is an arbitrary cell in $\Th[F]$.
\end{theorem}

\begin{proof}
For all $v_h\in \aXh^c$, gathering the sum by faces and using the flux conservativity \eqref{fv:cc:cons} shows that
\begin{align*}
\alinh(v_h)\coloneq\sum_{T\in\Th} \int_T f v_T{}&=\sum_{T\in\Th} \sum_{F\in\Fh[T]}\Fl^c(u_h)v_T\\
	{}&=\sum_{F\in\Fhi,\,\Th[F]=\{T,T'\}}\Fl^c(u_h)v_T+\Fl[T'F]^c(u_h)v_{T'}+
	\sum_{F\in\Fhb,\,\Th[F]=\{T\}}\Fl^c(u_h)v_T\\
	&=\sum_{F\in\Fhi,\,\Th[F]=\{T,T'\}}\Fl^c(u_h)(v_T-v_{T'})+
	\sum_{F\in\Fhb,\,\Th[F]=\{T\}}\Fl^c(u_h)v_T\\
	&=:\abilh(u_h,v_h).
\end{align*}
Hence, the cell-centred Finite Volume scheme \eqref{fv:balance.c} has been recast in the
framework of Section \ref{sec:abstract.analysis}. The error estimate \eqref{eq:fv.ener.est.cc}
then follows Theorem \ref{a:th.est.var}, in a similar way as the error estimate
\eqref{eq:fv.ener.est} for cell- and face-centred schemes.\end{proof}

As for cell- and face-centred FV methods, we could deduce from this theorem an error estimate
for schemes with local, bounded and linearly exact fluxes. However, some important FV methods are not linearly exact if the diffusion tensor $\diff$ is discontinuous. This is the case, for example, of Multi-Point Flux Approximation (MPFA) methods \cite{Edwards.Rogers:94,Aavatsmark.Barkve.ea:98,Aavatsmark:02}. To properly account for the diffusion jump, the fluxes are constructed to be exact on interpolants of \emph{piecewise} linear functions that have continuous fluxes (and, thus, usually discontinuous gradients to compensate for the discontinuity of the diffusion tensor involved in the fluxes). Theorem \ref{th:fv.ener.est.cc} however still yields energy error estimate for such methods. To illustrate this, we consider here the case of two Multi-Point Flux Approximation methods: the MPFA-L and MPFA-G methods.

\medskip

Let us first briefly present these two schemes (see \cite{Aavatsmark.Eigestad.ea:08,Agelas.Di-Pietro.ea:10} for the details). Here, $(\vec{x}_T)_{T\in\Th}$ are still free points in the cells, but $(\vec{x}_F)_{F\in\Fh}$ are the centers of mass of the faces. A \emph{group of faces} is any set of $d$ faces that belong to the same cell and share the same vertex, see Fig. \ref{fig-groups}. For each such group $G$ we fix a cell $T_G$ whose boundary contains all the faces in $G$; in most cases there is actually only one such cell, but for some non-convex cells there might situations with two possible choices for $T_G$ -- in which case we arbitrarily fix one choice (see Fig. \ref{fig-groups}, right).

\begin{figure}[!h]
\begin{center}
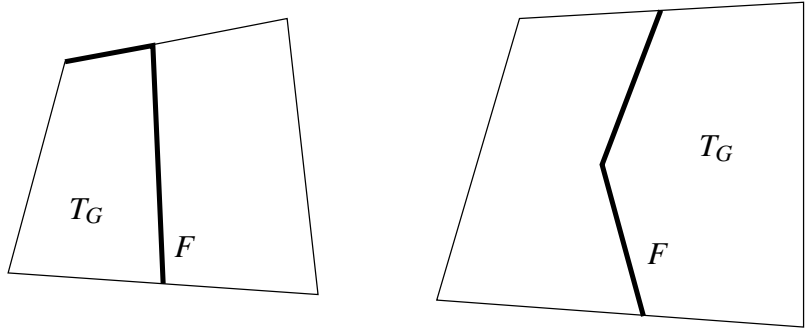
\caption{Two examples of groups of face (in bold) containing one particular face $F$. Left: unique choice for $T_G$; right: two
possible choices for $T_G$, one has been arbitrarily made.}
\label{fig-groups}
\end{center}
\end{figure}

The fluxes of $v_h\in \aXh^c$ are constructed via the notion of \emph{group gradients}. For a given group of faces $G$, let $\Th[G]$ be the set of cells that have at least one face in $G$. The group gradients
$\left\{(\GRAD_\disc v_h)_T^{G,F}\st T\in\Th[G]\,,\;F\in\Fh[T]\cap G\right\}\subset \Real^d$ are constructed by imposing the continuity of values and fluxes on all faces $F\in G$ of the piecewise linear functions having these gradients in each corresponding cells, and taking the values $(v_T)_{T\in\Th[G]}$ at $(\vec{x}_T)_{T\in\Th[G]}$. Additionally, for the cell $T_G$ previously selected, it is imposed that all group gradients are independent of the corresponding faces: $(\GRAD_\disc v_h)_{T_G}^{G,F}=(\GRAD_\disc v_h)_{T_G}^{G,F'}$ for all $F,F'\in \Fh[T]\cap G$. We denote by $(\GRAD_\disc v_h)_{T_G}^{G}$ the common value of all these group gradients associated with $T_G$, and it can be proved that this vector is a solution of the following linear system:
\begin{equation}
  \mathcal{A}_G (\GRAD_\disc v_h)_{T_G}^{G} = \mathcal{B}_G(v_h),
  \label{aG}
\end{equation}
where $\mathcal{A}_G\in\Real^{d \times d}$ is defined row-wise by 
\[
\mathcal{A}_G=
\begin{bmatrix}
  \left(
  \frac{\diff[T]\normal_{TF}\SCAL\normal_{TF}}{d_{T,F}}
  (\vec{x}_T-\vec{x}_{T_G})
  +\diff[T_G]\normal_{T_GF}+\diff[T]\normal_{TF}
  \right)_{F\in G\cap \Fhi}^t\\
  \left(
  \frac{\diff[T_G]\normal_{T_GF}\SCAL\normal_{T_GF}}{d_{T_G,F}}
  (\vec{x}_F-\vec{x}_{T_G})
  \right)_{F\in G\cap\Fhb}^t
\end{bmatrix},
\]
with $T$ the cell on the other side of $T_G$ with respect to $F$, and $\mathcal{B}_G(v)\in\Real^d$ is defined as
\[
\mathcal{B}_G(v_h)=
\begin{bmatrix}
  \left(
  \dfrac{\diff[T]\normal_{TF}\SCAL\normal_{TF}}{d_{T,F}}
  (v_T-v_{T_G})
  \right)_{F\in G\cap\Fhi}\\
  \left(
  \dfrac{\diff[T_G]\normal_{T_GF}\SCAL\normal_{T_GF}}{d_{T_G,F}}
  (-v_{T_G})
  \right)_{F\in G\cap \Fhb}
\end{bmatrix}.
\]
For a given face $F$, we denote by $\mathcal G_F$ the set of groups $G$ containing $F$ and such that $\mathcal A_G$ is invertible (it is assumed that $\mathcal G_F\not=\emptyset$ for all $F\in\Fh$).
The numerical fluxes are then defined as a convex combination of the fluxes corresponding to the group gradients: for $T\in\Th$, $F\in\Fh[T]$ and $v_h\in\aXh^c$,
\begin{equation}\label{fv:fluxes.mpfa}
\Fl^c(v_h)\coloneq \sum_{G\in\mathcal G_F}\theta_F^{G}\Fl^{c,G}(v_h)\quad\mbox{ with }\quad
\Fl^{c,G}(v_h)\coloneq -|F|\diff[T](\GRAD_\disc v_h)_T^{G,F}\SCAL\normal_{TF},
\end{equation}
where $(\theta_F^G)_{G\in\mathcal G_F}$ are the coefficients of the convex combination. The $L$-scheme corresponds to the case where, for each face, this convex combination has only one non-zero coefficient, chosen to maximise the monotonicity properties of the scheme. The $G$-scheme corresponds to a choice of coefficients that maximise the coercivity properties of the resulting scheme.

We now show that Theorem \ref{th:fv.ener.est.cc} yields the following error estimate. This estimate seems to be the first one for the MPFA-L and MPFA-G methods in the case of discontinuous permeability tensors; all previous estimates available in the literature have been derived under the assumption that $\diff\in C^1(\overline{\Omega})^{d\times d}$, see \cite{Droniou:14} and reference therein.

\begin{theorem}[Error estimate for the MPFA-L/G methods]
Let $\eta$ be such that
\[
\eta\ge\max_{T\in\Th,\,F\in\Fh[T]}\frac{h_T}{d_{TF}}\,,\quad
\eta\ge\max_{F\in\Fhi,\,\Th[F]=\{T,T'\}} \frac{d_{TF}}{d_{T'F}}
\quad\mbox{ and }\quad\eta\ge \max_{F\in\Fh}\sum_{G\in\mathcal G_F} \theta_F^G|\mathcal A_G^{-1}|,
\]
where $|\mathcal A_G^{-1}|$ is the induced Euclidean norm of $\mathcal A_G^{-1}$.
Assume that the solution $u$ to \eqref{eq:weak} belongs to $C(\overline{\Omega})$ and that
$u_{|\Omega_i}\in C^2(\overline{\Omega_i})$ for each $i\in\{1,\ldots,N_\Omega\}$.
Assume that the fluxes \eqref{fv:fluxes.mpfa} satisfy the coercivity property \eqref{fv:coer}, and
let $u_h$ be the solution to the MPFA-L/G scheme (that is, \eqref{fv:balance.c} with these fluxes).
Then
\[
\norm[1,\Th,c]{u_h-\aIh^c u}\lesssim \gamma^{-1} \norm[C^2]{u} h,
\]
where $\norm[C^2]{u}\coloneq \max_{i=1,\ldots,N_\Omega}\norm[C^2(\overline{\Omega_i})]{u}$
and the hidden constant in $\lesssim$ depends only on $\Omega$, $\eta$ and $\diff$.
\end{theorem}

\begin{remark}[About the coercivity]
In general, the coercivity of MPFA methods is not known, and numerical tests indicate that it
might actually fail for MPFA-O scheme on some very distorted meshes \cite[Section 3.3]{Droniou:14}.
However, for the MPFA-L/G schemes, an indicator can be designed that only requires to compute the eigenvalues of small systems, and that provides a sufficient condition for the methods to be coercive \cite[Lemma 3.4]{Agelas.Di-Pietro.ea:10}.
\end{remark}

\begin{proof}
\cite[Lemma 3.3]{Agelas.Di-Pietro.ea:10} shows that, for all $T\in\Th$, $F\in\Fh[T]$ and $G\in\mathcal G_F$, $|(\GRAD_\disc \aIh^c u)_T^{G,F}-\GRAD u(\vec{x}_T)|\lesssim \norm[C^2]{u}(1+|\mathcal A_G^{-1}|)h$ (in this lemma, the quantity $\norm[C^2]{u}$ does not explicitly appear but is hidden in a constant `$C_5$'; the proof however clearly shows that this constant depends linearly on $\norm[C^2]{u}$). By $C^2$ regularity of $u$ in the sub-domain $\Omega_i$ that contains $T$, we infer that 
\[
\sup_{\vec{x}\in F}|(\GRAD_\disc \aIh^c u)_T^{G,F}-\GRAD u(\vec{x})|\lesssim \norm[C^2]{u}(1+|\mathcal A_G^{-1}|)h.
\]
Hence,
\[
\left|-|F|\diff[T](\GRAD_\disc \aIh^c u)_T^{G,F}\SCAL\normal_{TF}+\int_F \diff[T]\GRAD u_{|T}\SCAL\normal_{TF}\right|\lesssim |F|\norm[C^2]{u}(1+|\mathcal A_G^{-1}|)h.
\]
Taking the convex combination weighted by $(\theta_F^G)_{G\in\mathcal G_F}$ of this inequality and recalling the definition of $\eta$ and \eqref{fv:fluxes.mpfa}, we infer that
\[
\left|\Fl^c(\aIh^c u) +\int_F \diff[T]\GRAD u_{|T}\SCAL\normal_{TF}\right|\lesssim |F|\norm[C^2]{u}h.
\]
The proof is completed by plugging this estimate into \eqref{eq:fv.ener.est.cc} and by noticing that
$d_F|F|\lesssim |T|+|T'|$ if $F\in\Fhi$ with $\Th[F]=\{T,T'\}$, and $d_F|F|\lesssim |T|$
if $F\in\Fhb$ with $\Th[F]=\{T\}$, so that $\sum_{F\in\Fh} d_F|F|\lesssim\sum_{T\in\Th}|T|=|\Omega|$.\end{proof}

\section{Conclusion}\label{sec:conclusion}

We developed an abstract analysis framework, in the spirit of Strang's second lemma, for approximations of linear PDEs in weak form. Contrary to Strang's lemma, the approximations can be written in fully discrete form, with test and trial spaces that are not spaces of functions -- and thus not manipulable together with the continuous test and trial spaces. The framework identifies a general consistency error that bounds, under an inf--sup condition, the discrete norm of the difference between the approximation solution and an interpolant of the continuous solution. We also established improved estimates in a weaker norm, using the Aubin--Nitsche trick.

This abstract framework was applied to two popular families of numerical methods for diffusion equations: conforming and non-conforming VEM, and cell-centred or cell- and face-centred Finite Volume methods. For each of these methods, we obtained energy error estimates that accurately track the local dependencies on the diffusion tensor, through local anisotropy ratios and diffusion magnitude. In both cases, such estimates seem to be entirely new. Optimal $L^2$ error estimates were also established for VEM in a unified setting.

To analyse the VEM schemes for the anisotropic diffusion model, optimal approximation properties of the  oblique elliptic projector on local polynomial spaces were established. These properties are of general interest to several high-order methods for diffusion equations on polytopal meshes.

The range of models and numerical techniques covered by the analysis framework goes beyond the examples above. Actually, an inspection of error bounds in some previous works show that they are based on estimations of terms that are (components of) the consistency error of our abstract setting. For example, in \cite[Theorem 10]{Di-Pietro.Droniou.ea:15}, robust error estimates for the HHO method applied to an advection--diffusion--reaction model are established by bounding terms $\term_1$, $\term_2$ and $\term_3$ that respectively correspond to the consistency errors of the diffusion component of the model, of the advection--reaction component, and of the weakly enforced (\emph{\`a la} Nitsche) boundary conditions. The analysis in \cite{Di-Pietro.Droniou.ea:15} was however carried out in an \emph{ad-hoc} setting, and not identified as part of a wider theory as done in this paper.


\section*{Acknowledgements}
The work of the first author was supported by Agence Nationale de la Recherche grants HHOMM (ANR-15-CE40-0005) and fast4hho (ANR-17-CE23-0019).
The work of the second author was partially supported by the Australian Government through the Australian Research Council's Discovery Projects funding scheme (project number DP170100605).
Fruitful discussions with Simon Lemaire (INRIA Lille - Nord Europe) are gratefully acknowledged.


\begin{taggedblock}{dg}

  \appendix
  
\section{Discontinuous Galerkin methods}\label{app:dg}

In this appendix we apply the abstract analysis framework to DG methods.
  While the error estimates in this section can be found in the literature (see, e.g., \cite{Di-Pietro.Ern.ea:08,Di-Pietro.Ern:12}), working in a fully discrete formulation entails in our opinion some relevant simplifications, as discussed hereafter.
  First, unlike \cite[Definition 1.31]{Di-Pietro.Ern:12}, we do not need to assume further regularity on the exact solution to define the notion of consistency for the DG scheme: the requirement that $u$ sits in (at least) $H^2(\Th)$ only appears in the error estimate of Theorem \ref{th:err.est.dg}.
  Second, we have a clear notion of the primal-dual consistency requirement that intervenes in proving optimal $L^2$ error estimates; this notion is present but not explicitly formulated in \cite[Section 4.2.4]{Di-Pietro.Ern:12}.

We consider here the model and notations introduced in Section \ref{model:setting}.

\subsection{The Symmetric Weighted Interior Penalty method}

We assume that the mesh belongs to a regular sequence in the sense of \cite[Section 1.4]{Di-Pietro.Ern:12}, and denote by $\varrho>0$ the corresponding regularity parameter.
Under this assumption it holds, in particular, that the diameter of an element and those of its faces are uniformly comparable; see \cite[Lemma 1.42]{Di-Pietro.Ern:12}.
For any $T\in\Th$ and any $F\in\Fh[T]$, we let
\begin{equation}\label{eq:sdiff.TF}
  \sdiff[TF]\coloneq(\diff[T]\normal_{TF})\SCAL\normal_{TF}.
\end{equation}
For any internal face $F$, we select an arbitrary but fixed ordering of the elements $T_1,T_2\in\Th$ that share $F$ and, for any function $v$ admitting a possibly double-valued trace on $F$, we define the following jump and weighted average operators:
\begin{equation}\label{eq:trace.operators}
  \jump{v}\coloneq v_{|T_1} - v_{|T_2},\qquad
  \wavg{v}\coloneq \omega_1v_{|T_1} + \omega_2v_{|T_2},
\end{equation}
with diffusion-dependent weights $\omega_1,\omega_2\in(0,1)$ such that
\begin{equation}\label{eq:weights}
  \omega_1+\omega_2=1\quad\mbox{ and }\quad
  \max\left(\omega_1\sqrt{\sdiff[T_1F]},\omega_2\sqrt{\sdiff[T_2F]}\right)
  \le\sqrt{\frac{\sdiff[F]}{2}}\mbox{ with }
  \sdiff[F]\coloneq\frac{2\sdiff[T_1F]\sdiff[T_2F]}{\sdiff[T_1F]+\sdiff[T_2F]}.
\end{equation}
The simplest possible choice for the weights is $\omega_1=1-\omega_2\coloneq\sqrt{\sdiff[T_2F]}/(\sqrt{\sdiff[T_1F]}+\sqrt{\sdiff[T_2F]})$.
As a matter of fact, in this case we have that
$$
\max\left(\omega_1\sqrt{\sdiff[T_1F]},\omega_2\sqrt{\sdiff[T_2F]}\right)
= \frac{\sqrt{\sdiff[T_1F]\sdiff[T_2F]}}{\sqrt{\sdiff[T_1F]}+\sqrt{\sdiff[T_2F]}}
\le\frac{\sqrt{\sdiff[T_1F]\sdiff[T_2F]}}{\sqrt{\sdiff[T_1F]+\sdiff[T_2F]}}
=\sqrt{\frac{\sdiff[F]}{2}}.
$$
We extend the above notations to any boundary face $F\in\Fhb$ such that $F\in\Fh[T]$ for some $T\in\Th$ by setting $\jump{v}=\wavg{v}\coloneq v$ and $\sdiff[F]\coloneq\sdiff[TF]$.

Let an integer $k\ge1$ be fixed.
The space of unknowns for DG methods is spanned by piecewise polynomial functions of total degree $\le k$, i.e., $\aXh=\aYh=\Poly{k}(\Th)$.
The associated interpolant $\aIh$ is the broken $L^2$-projector $\lproj{k}:L^2(\Omega)\to\Poly{k}(\Th)$, defined from the local $L^2$-projectors $\lproj[T]{k}$ by: For all $w\in L^2(\Omega)$,
\begin{equation}\label{def:Ih.dg}
(\lproj[h]{k}w)_{|T}= \lproj[T]{k}(w_{|T})\quad\forall T\in\Th.
\end{equation}
For the sake of simplicity, we focus on the so-called Symmetric Weighted Interior Penalty method, which is written under the form \eqref{a:pro.h} with
\begin{equation}\label{eq:source.dg}
\alinh(v_h)= (f,v_h)\qquad\forall v_h\in \aXh,
\end{equation}
and, for $w_h,v_h\in\aXh$,
\begin{equation}\label{eq:ah.dg}
  \begin{aligned}
    \abilh(w_h,v_h)
    &\coloneq (\diff\GRADh w_h,\GRADh v_h) + \eta s_h(w_h,v_h)
    \\
    &\quad-\sum_{F\in\Fh}\left[
    (\wavg{\diff\GRADh w_h}\SCAL\normal_F,\jump{v_h})_F
    + (\jump{w_h},\wavg{\diff\GRADh v_h}\SCAL\normal_F)_F
    \right]
  \end{aligned}
\end{equation}
with $\eta>0$ a user-dependent penalty parameter, and the stabilisation bilinear form $s_h$ defined by
\begin{equation}\label{eq:sh.dg}
  s_h(w_h,v_h)\coloneq\sum_{F\in\Fh}\sdiff[F]h_F^{-1}(\jump{w_h},\jump{v_h})_F.
\end{equation}

\begin{remark}[Variations on symmetry and stabilisation]
  Non-symmetric and skew-symmetric variations of this method, as well as alternative stabilisations bilinear forms, can be found in \cite{Di-Pietro.Ern.ea:08}.
  The following analysis can be extended to cover all these variations, but this topic will not be addressed here to keep the exposition as simple as possible.
We refer the interested reader to \cite[Section 5.3]{Di-Pietro.Ern:12} for an extensive discussion on this topic.
\end{remark}
We equip the space $X_h$ with the following norm, that involves volumetric gradients and jumps across the faces:
\begin{equation}\label{eq:norm.U.dg}
  \norm[\aXh]{v_h}\coloneq\left(
  \norm{\diff^{\frac12}\GRADh v_h}^2 + \seminorm[s,h]{v_h}^2
  \right)^{\frac12},\qquad
  \seminorm[s,h]{v_h}\coloneq 
  s_h(v_h,v_h)^{\frac12}.
\end{equation}
To establish the coercivity of $\abilh$, the parameter $\eta$ must be chosen large enough. This choice depends on two quantities. The first, $N_\partial\in\Natural$, is a bound on the number of faces of each mesh element. The second, $C_{\rm tr}>0$, is the constant of the following discrete trace inequality, valid for all $T\in\Th$, all $v\in\Poly{k}(T)$, and all $F\in\Fh[T]$:
\begin{equation}\label{eq:discrete.trace}
  \norm[F]{v}\le C_{\rm tr}h_F^{-\frac12}\norm[T]{v}.
\end{equation}
The mesh regularity assumptions \cite[Section 1.4]{Di-Pietro.Ern:12} ensure that $N_\partial$ and  $C_{\rm tr}$ can be bounded from above independently of $h$.

\subsection{Error estimate in energy norm}

\begin{theorem}[Energy estimate for the DG method]\label{th:err.est.dg}
Assume that $\eta>C_{\rm tr}^2N_\partial$ and let 
\begin{equation}\label{def:gamma.dg}
\gamma\coloneq\frac{\eta-C_{\rm tr}^2N_\partial}{1+\eta}.
\end{equation}
Let moreover $1\le r\le k$, and assume that the solution $u\in H^1_0(\Omega)$ to \eqref{eq:weak}
belongs to $H^{r+1}(\Th)$.
Then, denoting by $u_h$ the solution of the SWIP DG scheme, it holds that
\begin{equation}\label{dg:est.energy}
\norm[\aXh]{u_h-\lproj[h]{k} u}\lesssim \gamma^{-1}(1+\eta)\left(
      \sum_{T\in\Th}\overline{\lambda}_T h_T^{2r}\seminorm[H^{r+1}(T)]{u}^2
      \right)^{\frac12}
\end{equation}
where the hidden constant is independent of $h$, $\diff$, and $\eta$.
\end{theorem}
Before proving Theorem \ref{th:err.est.dg}, we state the following preliminary lemma, which contains an estimate of the boundary terms in $a_h$.
Its proof is a straightforward adaptation of the arguments of \cite[Lemma 4.50]{Di-Pietro.Ern:12}.
\begin{lemma}[Estimate of boundary terms]
  It holds, for all $w\in H^2(\Th)$ and all $v_h\in \aXh$
  \begin{equation}\label{eq:consistency.term.est}
    \left|
    \sum_{F\in\Fh}(\wavg{\diff\GRADh w}\SCAL\normal_F,\jump{v_h})_F
    \right|\le
    \left(
    \sum_{T\in\Th} h_T\norm[\partial T]{\diff[T]^{\frac12}\GRAD w_{|T}}^2
    \right)^{\frac12}\seminorm[s,h]{v_h}
  \end{equation}
  so that in particular, for all $w_h\in \aXh$ and $v_h\in \aXh$,
  \begin{equation}\label{eq:symmetry.term.est}
    \left|
    \sum_{F\in\Fh}(\wavg{\diff\GRADh w_h}\SCAL\normal_F,\jump{v_h})_F
    \right|\le
    C_{\rm tr} N_\partial^{\frac12}\norm{\diff^{\frac12}\GRADh w_h}\seminorm[s,h]{v_h}.
  \end{equation}
\end{lemma}
\begin{proof}
  (i) \emph{Proof of \eqref{eq:consistency.term.est}.}
  For any $F\in\Fhb$ such that $F\in\Fh[T]$ for some $T\in\Th$, we can readily write
  \begin{equation}\label{eq:consistency.term.est:boundary.face}
    \begin{aligned}
      |(\wavg{\diff\GRAD w_{|T}}\SCAL\normal_F,\jump{v_h})_F|
      &= |(\diff[T]\GRAD w_{|T}\SCAL\normal_F,\jump{v_h})_F|
      \\
      &\le\norm[F]{\diff[T]^{\frac12}\GRAD w_{|T}} \norm[F]{|\diff[T]^{\frac12}\SCAL\normal_F|\jump{v_h}}
      \\
      &\le h_F^{\frac12}\norm[F]{\diff[T]^{\frac12}\GRAD w_{|T}}
      ~\left(\frac{\sdiff[F]}{h_F}\right)^{\frac12}\norm[F]{\jump{v_h}},
    \end{aligned}
  \end{equation}
  where we have used the definition of the weighted average operator on boundary faces in the first line and the symmetry of $\diff[T]$ together with a Cauchy--Schwarz inequality in the second line.
  To pass to the third line, we have multiplied and divided by $h_F^{1/2}$, we have extracted the constant scalar $|\diff[T]^{1/2}\normal_F|$ from the norm, and we have observed that $|\diff[T]^{1/2}\normal_F|^2=\diff[T]^{1/2}\normal_F\SCAL\diff[T]^{1/2}\normal_F=\diff[T]\normal_F\SCAL\normal_F=\sdiff[F]$.

  For an internal face $F\in\Fhi$ shared by distinct mesh elements $T_1,T_2\in\Th$ we have, on the other hand,
  \begin{align}
      |(\wavg{\diff\GRADh w}\SCAL\normal_F,\jump{v_h})_F|
      &\le
      |\omega_1(\diff[T_1]\GRAD w_{|T_1}\SCAL\normal_F,\jump{v_h})_F|
      + |\omega_2(\diff[T_2]\GRAD w_{|T_2}\SCAL\normal_F,\jump{v_h})_F|
      \nonumber\\
      &\le
      \omega_1\sdiff[T_1F]^{\frac12}\norm[F]{\diff[T_1]^{\frac12}\GRAD w_{|T_1}}\norm[F]{\jump{v_h}}
      + \omega_2\sdiff[T_2F]^{\frac12}\norm[F]{\diff[T_2]^{\frac12}\GRAD w_{|T_2}}\norm[F]{\jump{v_h}}
      \nonumber\\
      &\le\frac1{\sqrt{2}}\left(\sum_{T\in\Th[F]}h_F^{\frac12}\norm[F]{\diff[T]^{\frac12}\GRAD w_{|T}}\right)
      ~\left(\frac{\sdiff[F]}{h_F}\right)^{\frac12}\norm[F]{\jump{v_h}},
      \nonumber\\
      &\le\left(\sum_{T\in\Th[F]}h_F\norm[F]{\diff[T]^{\frac12}\GRAD w_{|T}}^2\right)^{\frac12}
      ~\left(\frac{\sdiff[F]}{h_F}\right)^{\frac12}\norm[F]{\jump{v_h}},
		\label{eq:consistency.term.est:internal.face}
  \end{align}
  where we have used the definition \eqref{eq:trace.operators} of the weighted average operator together with the triangle inequality in the first line,
  we have proceeded as in \eqref{eq:consistency.term.est:boundary.face} for each addend to pass to the second line,
  we have multiplied and divided by $h_F^{1/2}$ and used the properties \eqref{eq:weights} of the weights to pass to the third line,
  and we have used a discrete Cauchy--Schwarz inequality on the sum over $T\in\Th[F]$ to conclude.
  
  Using \eqref{eq:consistency.term.est:internal.face} and \eqref{eq:consistency.term.est:boundary.face} to estimate the argument of the summation in the left-hand side of \eqref{eq:consistency.term.est} followed by a discrete Cauchy--Schwarz inequality on the sum over $F\in\Fh$, we arrive at
  \begin{equation}\label{eq:consistency.term.est:final}
    \begin{aligned}
      \left|  
      \sum_{F\in\Fh}(\wavg{\diff\GRADh w}\SCAL\normal_F,\jump{v_h})_F
      \right|
      &\le\left(
      \sum_{F\in\Fh}\sum_{T\in\Th[F]}h_F\norm[F]{\diff[T]^{\frac12}\GRAD w_{|T}}^2
      \right)^{\frac12}\seminorm[s,h]{v_h}
      \\
      &\le\left(
      \sum_{T\in\Th}h_T\norm[\partial T]{\diff[T]^{\frac12}\GRAD w_{|T}}^2
      \right)^{\frac12}\seminorm[s,h]{v_h},
    \end{aligned}
  \end{equation}
  where, to pass to the second line, we have used the fact that
  \begin{equation}\label{eq:sum.swap}
    \sum_{F\in\Fh}\sum_{T\in\Th[F]}\bullet = \sum_{T\in\Th}\sum_{F\in\Fh[T]}\bullet
  \end{equation}
  followed by the fact that, for any $T\in\Th$, $\partial T = \bigcup_{F\in\Fh[T]}\overline{F}$ and $h_F\le h_T$ for any $F\in\Fh[T]$.
  \medskip\\
  \emph{Proof of \eqref{eq:symmetry.term.est}.}
  Resuming from the first line of \eqref{eq:consistency.term.est:final} with $w=w_h$, using the discrete trace inequality \eqref{eq:discrete.trace} together with \eqref{eq:sum.swap}, we obtain
  \begin{align*}
  \left|  
  \sum_{F\in\Fh}(\wavg{\diff\GRADh w_h}\SCAL\normal_F,\jump{v_h})_F
  \right|
  \le{}&\left(
  \sum_{T\in\Th}\sum_{F\in\Fh[T]} C_{\rm tr}^2\norm[T]{\diff[T]^{\frac12}\GRAD w_h}^2
  \right)^{\frac12}\seminorm[s,h]{v_h}\\
  \le{}& C_{\rm tr} N_\partial^{\frac12}\norm{\diff^{\frac12}\GRADh w_h}\seminorm[s,h]{v_h}.\qedhere
  \end{align*}
\end{proof}
\begin{proof}[Proof of Theorem \ref{th:err.est.dg}]
  We first establish some coercivity and consistency properties of the DG scheme, before concluding by invoking Theorems \ref{a:th.est.var}.
  \medskip\\
  (i) \emph{Coercivity.}
  Making $w_h=v_h$ in \eqref{eq:ah.dg}, and using \eqref{eq:symmetry.term.est} to estimate the terms in the second line of \eqref{eq:ah.dg}, it is readily inferred that
  $
  \abilh(v_h,v_h)
  \ge\norm{\diff^{\frac12}\GRADh v_h}^2
  + \eta\seminorm[s,h]{v_h}^2
  - 2C_{\rm tr}N_\partial^{\frac12}\norm{\diff^{\frac12}\GRADh v_h}\seminorm[s,h]{v_h}.
  $
  Invoking the inequality $x^2+\eta y^2-2\beta xy\ge\frac{\eta-\beta^2}{1+\eta}(x^2+y^2)$, valid for all $x,y\in\Real$, all $\beta>0$, and all $\eta>\beta^2$, with $x=\norm{\diff^{\frac12}\GRADh v_h}$, $y=\seminorm[s,h]{v_h}$, and $\beta=C_{\rm tr}N_\partial^{\frac12}$, we conclude recalling the definition \eqref{def:gamma.dg} of $\gamma$ that
  \begin{equation}\label{eq:coer.dg}
    \abilh(v_h,v_h)\ge \gamma \norm[\aXh]{v_h}^2.
  \end{equation}
  \medskip\\
  (ii) \emph{Primal consistency.}
  The definitions \eqref{var.cons} and \eqref{eq:source.dg} of the primal consistency error $\Cerr{u}{\cdot}$ and of the DG linear form show that, for all $v_h\in\aXh$,
  \begin{align*}
  \Cerr{u}{v_h}={}&
  -(\DIV(\diff\GRAD u),v_h)-\abilh(\lproj[h]{k} u,v_h)\\
  ={}&(\diff\GRAD u,\GRADh v_h)
  - \sum_{F\in\Fh}(\diff\GRAD u\SCAL\normal_F,\jump{v_h})_F-\abilh(\lproj[h]{k} u,v_h),
  \end{align*}
  where the second equality follows by integrating-by-parts in each element.
  Invoking the definition \eqref{eq:ah.dg} of $\abilh$ and using the fact that $\diff\GRAD u\SCAL\normal_F=\wavg{\diff\GRAD u}\SCAL\normal_F$ for all $F\in\Fh$ (since $\diff\GRAD u\SCAL \normal$ is continuous across the interfaces) we obtain
  \begin{equation}\label{eq:consistency.T1-T4}
    \begin{aligned}
      \Cerr{u}{v_h}
      ={}& (\diff\GRADh(u-\lproj[h]{k} u),\GRADh v_h)
      - \sum_{F\in\Fh}(\wavg{\diff\GRADh(u-\lproj[h]{k} u)}\SCAL\normal_F,\jump{v_h})_F
      \\
      & - \sum_{F\in\Fh}(\jump{\lproj[h]{k} u},\wavg{\diff\GRADh v_h}\SCAL\normal_F)_F
      + \eta\sum_{F\in\Fh}\frac{\sdiff[F]}{h_F}(\jump{\lproj[h]{k} u},\jump{v_h})_F
      \\
      \eqcolon{}&\term_1+\cdots+\term_4.
    \end{aligned}
  \end{equation}
  To estimate $\term_1$, we use the Cauchy--Schwarz inequality, the definition \eqref{def:Ih.dg} of the interpolant $\lproj[h]{k}$, the optimal approximation properties \eqref{eq:lproj:approx} of the $L^2$-projector with $\polydeg=k$, $m=1$, and $s=r+1$, and the definition \eqref{eq:norm.U.dg} of the norm $\norm[\aXh]{{\cdot}}$ to write
  \begin{equation}\label{eq:consistency.dg:T1}
    \begin{aligned}
      |\term_1|
      &\le
      \left(
      \sum_{T\in\Th}  \norm[T]{\diff^{\frac12}(\GRAD u-\GRAD\lproj[T]{k}u_{|T})}^2
      \right)^{\frac12}\left(
      \sum_{T\in\Th}\norm[T]{\diff^{\frac12}\GRAD v_T}^2
      \right)^{\frac12}
      \\
      &\lesssim
      \left(\sum_{T\in\Th} \overline{\lambda}_Th_T^{2r}\seminorm[H^{r+1}(T)]{u}^2\right)^{\frac12}\norm[\aXh]{v_h}.
    \end{aligned}
  \end{equation}
  Using the bound \eqref{eq:consistency.term.est} followed by the trace approximation properties \eqref{eq:lproj:approx.trace} of the $L^2$-projector with $\polydeg=k$, $m=1$, and $s=r+1$, we have for the second term
  \begin{equation}\label{eq:consistency.dg:T2}
    |\term_2|\le\left(
    \sum_{T\in\Th}h_T\norm[\partial T]{\diff[T]^{\frac12}\GRAD(u_{|T}-\lproj[T]{k}u_{|T})}^2
    \right)^{\frac12}\seminorm[s,h]{v_h}
    \lesssim\left(
    \sum_{T\in\Th}\overline{\lambda}_T h_T^{2r}\seminorm[H^{r+1}(T)]{u}^2
    \right)^{\frac12}\seminorm[s,h]{v_h}.
  \end{equation}
  The following estimate will be used to bound the remaining terms:
  For all $w\in H^1_0(\Omega)\cap H^{r+1}(\Th)$,
  \begin{align}
    \seminorm[s,h]{\lproj{k}w}
    &=\left(
    \sum_{F\in\Fh}\sdiff[F]h_F^{-1}\norm[F]{\jump{\lproj{k}w}}^2
    \right)^{\frac12}
    \nonumber\\
    &\lesssim\left(
    \sum_{F\in\Fh}\sum_{T\in\Th[F]}\overline{\lambda}_Th_T^{-1}
    \norm[F]{\lproj[T]{k}w_{|T}-w_{|T}}^2
    \right)^{\frac12}
    \nonumber\\
    &\lesssim\left(
    \sum_{T\in\Th}\overline{\lambda}_Th_T^{-1}\norm[\partial T]{\lproj[T]{k}w_{|T}-w_{|T}}^2    
    \right)^{\frac12}
    \lesssim\left(
    \sum_{T\in\Th}\overline{\lambda}_T h_T^{2r}\seminorm[H^{r+1}(T)]{w}^2
    \right)^{\frac12},
    \label{eq:consistency.dg.jumps}
  \end{align}
  where we have used the definition \eqref{eq:norm.U.dg} in the first line, inserted $\pm w$ inside the jump operator and used the triangle inequality together with the fact that $\sdiff[F]h_F^{-1}\lesssim\overline{\lambda}_Th_T^{-1}$ in the second line, used \eqref{eq:sum.swap} to pass to the third line, and invoked the optimal trace approximation properties \eqref{eq:lproj:approx.trace} of the $L^2$-projector with $\polydeg=k$, $m=0$, and $s=r+1$ to conclude.
  Notice that inserting $\pm w$ is possible in the second line since, by the assumed regularity, $\jump{w}=0$ across interfaces and $w=0$ on $\partial\Omega$.

  The bounds \eqref{eq:symmetry.term.est} and \eqref{eq:consistency.dg.jumps} with $w=u$ then yield
  \begin{equation}\label{eq:consistency.dg:T3}
    |\term_3| \le C_{\rm tr}N_\partial^2\norm{\diff^{\frac12}\GRADh v_h} \seminorm[s,h]{\lproj{k}u}
    \lesssim\norm[\aXh]{v_h}\left(
    \sum_{T\in\Th}\overline{\lambda}_Th_T^{2r}\seminorm[H^{r+1}(T)]{u}^2
    \right)^{\frac12}.
  \end{equation}
  The bound on $\term_4$ is obtained by invoking the Cauchy--Schwarz inequality and \eqref{eq:consistency.dg.jumps} with $w=u$:
  \begin{equation}\label{eq:consistency.dg:T4}
    |\term_4|
    \le\eta\seminorm[s,h]{\lproj{k}u}\seminorm[s,h]{v_h}
    \lesssim\eta\left(
    \sum_{T\in\Th}\overline{\lambda}_Th_T^{2r}\seminorm[H^{r+1}(T)]{u}^2
    \right)^{\frac12}\seminorm[s,h]{v_h}.
  \end{equation}
  Using \eqref{eq:consistency.dg:T1}, \eqref{eq:consistency.dg:T2}, \eqref{eq:consistency.dg:T3}, and \eqref{eq:consistency.dg:T4} in \eqref{eq:consistency.T1-T4}, and noticing that $\seminorm[s,h]{v_h}\le \norm[\aXh]{v_h}$, we obtain the following bound on the dual norm of $\Cerr{u}{\cdot}$:
  \begin{equation}\label{eq:primal.consistency.dg}
    \norm[\aXh^\star]{\Cerr{u}{\cdot}}\lesssim(1+\eta)\left(
    \sum_{T\in\Th}\overline{\lambda}_T h_T^{2r}\seminorm[H^{r+1}(T)]{u}^2
    \right)^{\frac12}.
  \end{equation}
  (iii) \emph{Conclusion.} 
  Using the coercivity property \eqref{eq:coer.dg} and the primal consistency estimate \eqref{eq:primal.consistency.dg}, the energy estimate \eqref{dg:est.energy} follows directly from Estimate \eqref{energy.est} in Theorem \ref{a:th.est.var}.
\end{proof}

\begin{remark}[Spatially varying diffusion field]\label{rem:dg.spatially.varying.K}
  The SWIP DG method, as well as the analysis above, can be extended to locally varying diffusion fields, at the price of a less favourable dependence of the constant in \eqref{dg:est.energy} on the latter.
  The modifications are briefly described hereafter.
  Assume $\diff\in L^\infty(\Omega)^{d\times d}\cap W^{1,\infty}(\Th)^{d\times d}$ symmetric and uniformly elliptic so that, in particular, there exists a symmetric and uniformly elliptic field $\diff^{-1}\in L^\infty(\Omega)^{d\times d}$ such that, denoting by $\matr{I}_d$ the identity matrix of $\Real^{d\times d}$, $\diff\diff^{-1}=\matr{I}_d$ almost everywhere in $\Omega$.
  For all $T\in\Th$, we let $\underline{\lambda}_T\coloneq1/\norm[L^\infty(T)^{d\times d}]{\diff^{-1}}$ and $\overline{\lambda}_T\coloneq\norm[L^\infty(T)^{d\times d}]{\diff}$.
  As before, we define for all $T\in\Th$ the quantity $\alpha_T\coloneq\overline{\lambda}_T/\underline{\lambda}_T$, which now represents a local anisotropy/heterogeneity ratio, and we set $\alpha\coloneq\max_{T\in\Th}\alpha_T$.
  Definition \eqref{eq:sdiff.TF} is modified setting $\sdiff[TF]\coloneq\norm[L^\infty(F)]{({\diff}_{|T}\normal_{TF})\SCAL\normal_{TF}}$.
  In estimate \eqref{eq:symmetry.term.est}, an additional factor $\alpha^{\frac12}$ appears in the right-hand side, so that \eqref{def:gamma.dg} is replaced by $\gamma\coloneq(\eta-\alpha C_{\rm tr}^2N_\partial)/(1+\eta)$, and $\eta$ must now be taken strictly larger than $\alpha C_{\rm tr}^2N_\partial$ for coercivity to hold.
  With these modifications, the proof of Theorem \ref{th:err.est.dg} carries out unchanged, but an additional dependence on the global anisotropy/heterogeneity ratio appears in the right-hand side of \eqref{dg:est.energy} as a result of the modification in \eqref{eq:symmetry.term.est}.
\end{remark}

\subsection{Improved error estimate in the $L^2$ norm}

\begin{theorem}[$L^2$ estimate for the DG method]\label{th:err.est.dg.L2}
  Under the assumptions and notations of Theorem \ref{th:err.est.dg}, and further assuming elliptic regularity, it holds that
  \begin{equation}\label{dg:est.L2}
    \norm{u_h-\lproj[h]{k} u}
    \lesssim \gamma^{-1}(1+\eta)^2 h^{r+1}\seminorm[H^{r+1}(\Th)]{u},
  \end{equation}
where the multiplicative constant additionally depends on $\diff$.
\end{theorem}

\begin{remark}[$L^2$-estimate for spatially varying diffusion field]\label{rem:dg.L2}
    As for VEM, and following the discussion in Remark \ref{rem:vem.L2}, we do not attempt to track the dependence on the diffusion field for the multiplicative constant in the right-hand side of \eqref{dg:est.L2}.
    Notice, however, that, in view of Remark \ref{rem:dg.spatially.varying.K}, the error estimate \eqref{dg:est.L2} can this time be extended to spatially varying diffusion fields.
\end{remark}
\begin{proof}
  (i) \emph{Dual consistency.}
Estimate \eqref{eq:primal.consistency.dg} also gives an estimate on the dual consistency error. Indeed, the operator $r_h:\aXh\to L^2(\Omega)$ here is the natural embedding of $\aXh=\Poly{k}(\Th)$ into $L^2(\Omega)$. Hence, inspecting the definition \eqref{var.cons.dual} of the dual consistency error and using the symmetry of $\abilh$ yields $\Cerrdual{z_g}{\cdot}=\Cerr{z_g}{\cdot}$. Applying \eqref{eq:primal.consistency.dg} with $r=1$ and using the elliptic regularity therefore gives, for any $g\in L^2(\Omega)$,
\begin{equation}\label{est:dual.cons.dg}
  \norm[\aXh^\star]{\Cerrdual{z_g}{\cdot}}\lesssim \overline{\lambda}_\Omega(1+\eta) h\seminorm[H^{2}(\Omega)]{z_g},
\end{equation}
where $\overline{\lambda}_\Omega=\max_{T\in\Th}\overline{\lambda}_T$.
\medskip\\
  (ii) \emph{Primal-dual consistency.}
For $g\in L^2(\Omega)$, we want to estimate $\Cerr{u}{\lproj[h]{k}z_g}$, which is given by \eqref{eq:consistency.T1-T4} with $v_h=\lproj[h]{k}z_g$. Note that, by the assumed elliptic regularity,
$z_g\in H^2(\Omega)$ and $\norm[H^2(\Omega)]{z_g}\lesssim \norm{g}$.
Considering \eqref{eq:consistency.dg:T2} and \eqref{eq:consistency.dg:T4} with $v_h=\lproj[h]{k}z_g$
and using \eqref{eq:consistency.dg.jumps} with $w=z_g$ and $r=1$, we see that
\begin{equation}\label{eq:pd.consistency.1}
|\term_2|+|\term_4|\lesssim (1+\eta) \left(\overline{\lambda}_\Omega^{\frac12} h^r\seminorm[H^{r+1}(\Th)]{u} \right)\left(\overline{\lambda}_\Omega^{\frac12}h\seminorm[H^2(\Omega)]{z_g}\right).
\end{equation}
To estimate $\term_1+\term_3$, we first manipulate $\term_1$ as follows:
\begin{align*}
\term_1={}&(\diff\GRADh(u-\lproj[h]{k} u),\GRADh \lproj[h]{k}z_g)\nonumber\\
	={}& (\diff\GRADh(u-\lproj[h]{k} u),\GRADh \lproj[h]{k}z_g-\GRAD z_g) +(\diff\GRADh(u-\lproj[h]{k} u),\GRAD z_g) \\
	={}& (\diff\GRADh(u-\lproj[h]{k} u),\GRADh \lproj[h]{k}z_g-\GRAD z_g) + \sum_{F\in\Fh}(\jump{u-\lproj[h]{k} u},\diff\GRAD z_g\SCAL\normal_F)_F
+(u-\lproj[h]{k}u,g),
\end{align*}
where we passed to the second line by introducing $\pm\GRAD z_g$ and to the third line by performing a cell-wise integration-by-parts and recalling that $-\DIV (\diff\GRAD z_g)=g$. Note that fluxes $\diff\GRAD z_g\SCAL\normal_F$ are continuous across $F$, and thus $\diff\GRAD z_g\SCAL\normal_F=\wavg{\diff\GRAD z_g}\SCAL\normal_F$. Recalling the definition of $\term_3$ and noticing that, since $u\in H^1_0(\Omega)$, $\jump{u}=0$ for all $F\in\Fh$, this yields
\begin{align*}
\term_1+\term_3	={}& (\diff\GRADh(u-\lproj[h]{k} u),\GRADh \lproj[h]{k}z_g-\GRAD z_g) \\
&+ \sum_{F\in\Fh}(\jump{u-\lproj[h]{k} u},\wavg{\diff(\GRAD z_g-\GRADh \lproj[h]{k}z_g)}\SCAL\normal_F)_F
+(u-\lproj[h]{k}u,g).
\end{align*}
The Cauchy--Schwarz inequality and the optimal approximation properties of $\lproj[h]{k}$ that descend from \eqref{eq:lproj:approx}--\eqref{eq:lproj:approx.trace} yield
\begin{align*}
|(\diff\GRADh(u-\lproj[h]{k} u){}&,\GRADh \lproj[h]{k}z_g-\GRAD z_g)|+|(u-\lproj[h]{k}u,g)|\\
\le{}& \overline{\lambda}_\Omega \norm{\GRADh(u-\lproj[h]{k} u)}
\norm{\GRADh(z_g-\lproj[h]{k}z_g)}+\norm{u-\lproj[h]{k}u}\norm{g}\\
\lesssim{}& \overline{\lambda}_\Omega h^{r}\seminorm[H^{r+1}(\Th)]{u} h
\seminorm[H^2(\Omega)]{z_g}+h^{r+1}\seminorm[H^{r+1}(\Th)]{u}\norm{g}.
\end{align*}
For a given $F\in\Fh$, splitting the jump and average into their values in the cells neighbouring $F$ and using the trace approximation properties \eqref{eq:lproj:approx.trace} with $\polydeg=k$, $m=0$, and $s=r+1$ for the first factor, and with $\polydeg=k$, $m=1$, and $s=2$ for the second factor, we have
\begin{align*}
\Big|(\jump{u-\lproj[h]{k} u},{}&\wavg{\diff(\GRAD z_g-\GRADh \lproj[h]{k}z_g)}\SCAL\normal_F)_F\Big|\\
\le{}& \left(\sum_{T\in\Th[F]} \norm[F]{u_{|T}-\lproj[T]{k}u_{|T}}\right)
\left(\sum_{T\in\Th[F]} \overline{\lambda}_{T}\norm[F]{\GRAD (z_g)_{|T}-\GRAD\lproj[T]{k}(z_g)_{|T}}\right)\\
\lesssim{}& \left(\sum_{T\in\Th[F]} h_{T}^{r+\frac12}\seminorm[H^{r+1}(T)]{u}\right)
\left(\sum_{T\in\Th[F]} \overline{\lambda}_{T} h_{T}^{\frac12}\seminorm[H^2(T)]{z_g}\right).
\end{align*}
Summing over $F\in\Fh[T]$ and using, thanks to the Cauchy--Schwarz inequality,
the estimate (for positive terms $\bullet$ and $\star$)
\begin{align*}
\sum_{F\in\Fh}\Bigg(\sum_{T\in\Th[F]}\bullet\Bigg)\Bigg(\sum_{T\in\Th[F]}\star\Bigg)
\le{}&
\left(\sum_{F\in\Fh}\Bigg(\sum_{T\in\Th[F]}\bullet\Bigg)^2\right)^{\frac12}
\left(\sum_{F\in\Fh}\Bigg(\sum_{T\in\Th[F]}\star\Bigg)^2\right)^{\frac12}\\
\le{}&
2\left(\sum_{F\in\Fh}\sum_{T\in\Th[F]}\bullet^2\right)^{\frac12}
\left(\sum_{F\in\Fh}\sum_{T\in\Th[F]}\star^2\right)^{\frac12}
\end{align*}
(where the second inequality follows from ${\rm Card}(\Th[F])\le 2$ for all $F\in\Fh$),
as well as \eqref{eq:sum.swap}, we infer that
\[
\left|\sum_{F\in\Fh}(\jump{u-\lproj[h]{k} u},\wavg{\diff(\GRAD z_g-\GRADh \lproj[h]{k}z_g)}\SCAL\normal_F)_F\right|\lesssim h^{r+\frac12}\seminorm[H^{r+1}(\Th)]{u}\overline{\lambda}_\Omega
h^{\frac12}\seminorm[H^2(\Omega)]{z_g}.
\]
The previous estimates show that
\begin{equation}\label{eq:pd.consistency.2}
|\term_1+\term_3|\lesssim \overline{\lambda}_\Omega h^{r+1}\seminorm[H^{r+1}(\Th)]{u}\seminorm[H^2(\Omega)]{z_g}+ h^{r+1}\seminorm[H^{r+1}(\Th)]{u}\norm{g}.
\end{equation}
Gathering \eqref{eq:pd.consistency.1} and \eqref{eq:pd.consistency.2} and using the elliptic
regularity shows that
\begin{equation}\label{eq:pd.consistency.final}
	\Cerr{u}{\lproj[h]{k}z_g}\lesssim (1+\eta) h^{r+1}
\seminorm[H^{r+1}(\Th)]{u}\norm{g}.
\end{equation}
(iii) \emph{Conclusion.} 
Recalling that $r_h$ is here the natural embedding of $\Poly{k}(\Th)$ into $L^2(\Omega)$,
the $L^2$-estimate \eqref{dg:est.L2} follows from Theorem \ref{th:abstract.weak}, using
\eqref{dg:est.energy}, \eqref{est:dual.cons.dg} and \eqref{eq:pd.consistency.final}.
\end{proof}

\end{taggedblock}


\begin{footnotesize}
  \bibliographystyle{plain}
  \bibliography{strang}
\end{footnotesize}

\end{document}